\documentclass[11pt]{article}      
\usepackage[T1]{fontenc}           
\usepackage[utf8]{inputenc}        
\usepackage{graphicx}      
\usepackage[english]{babel} 
\usepackage{geometry}              
\usepackage{mathptmx}              
\usepackage{amsmath}               
\usepackage{caption}
\usepackage{subcaption}
\usepackage{verbatim} 
\usepackage{amssymb}
\usepackage{amsfonts}
\usepackage{amsthm}
\usepackage{cancel} 
\usepackage{comment}
\usepackage{hyperref} 
\usepackage{empheq} 

\usepackage{dsfont}
\usepackage{csquotes}
\usepackage{tikz}
\usepackage{tkz-tab} 
\usepackage{array}
\usepackage{float} 

\usepackage{empheq} 
\usepackage{enumerate} 
\usepackage[dvipsnames]{xcolor}
\definecolor{my_green}{RGB}{34,139,34}

\usepackage{tikz}
\usepackage{subcaption}
\usepackage{authblk}

\newcommand{\dt}{\partial_{t\,}}
\newcommand{\dx}{\partial_{x\,}}
\newcommand{\dy}{\partial_{y\,}}

\newcommand{\RR}{\mathbb{R}}
\newcommand{\NN}{\mathbb{N}}
\newcommand{\ZZ}{\mathbb{Z}}
\DeclareMathOperator{\dive}{div}

\newtheorem{assumption}{Assumption}[section]
\newtheorem{theorem}{Theorem}[section]
\newtheorem{proposition}{Proposition}[section]
\newtheorem{lemma}{Lemma}[section]
\newtheorem{corollary}{Corollary}[section]
\newtheorem{remark}{Remark}[section]
\newtheorem{definition}{Definition}[section]

\title{A Particle method for stationary transport equations}
\date{}

\author[1]{Rafael Bailo}
\author[2,3]{Julie Binard}
\author[3]{Pierre Degond}
\author[2]{Pascal Noble}

\affil[1]{Department of Mathematics and Computer Science\\
Eindhoven University of Technology; Eindhoven, Netherlands\\
email: r.bailo@tue.nl\\
}
\affil[2]{Institut de Mathématiques de Toulouse; UMR 5219\\ 
Université de Toulouse; CNRS \\
INSA, F-31077 Toulouse, France\\
email JB: binard@insa-toulouse.fr\\
email PN: noble@insa-toulouse.fr
}
\affil[3]{Departamento de Matemática Aplicada I\\
Universidad de Sevilla, E.T.S. Arquitectura. Avda Reina Mercedes, 41012 Sevilla, Spain\\
email: jbinard@us.es\\
}
\affil[4]{Institut de Mathématiques de Toulouse; UMR 5219\\ 
Université de Toulouse; CNRS \\
UPS, F-31062 Toulouse Cedex 9, France\\
email: pierre.degond@math.univ-toulouse.fr
}

\begin{document}

\maketitle

\begin{abstract}
We present and study a Particle method for the stationary solutions of a class of transport equations. This method is inspired by non-stationary Particle methods, the time variable being replaced by one spatial variable. Particles trajectories are computed using the ``time-dependent'' equations, and then the approximation is based on a quadrature method using the particle locations as quadrature points. We prove the convergence of the scheme under suitable regularity assumptions on the data and the solution, together with a ``characteristic completeness'' assumption (the characteristic curves fullfill the whole computational domain). We also provide an error estimate. The scheme is tested numerically on a two dimensional linear equation and we present a numerical study of convergence. Finally, we use this method to carry out numerical simulations of a landscape evolution model, where an erodible topography evolves under the effects of water erosion and sedimentation. The scheme is then useful to deal with wet/dry areas.
\end{abstract}

  

\section{Introduction}

In this paper we introduce and analyze a Particle method for first order equations in divergence form and with a source term:
\begin{equation}
    \dive({\bf a} u)(x) + a_0(x) u(x) = S(x),\quad \forall x\in \RR^d_+ := (0,+\infty) \times \RR^{d-1},
    \label{eq_system}
\end{equation}
with $d \geq 2$. We assume that the functions ${\bf a}, \, a_0, \, S$ are defined on an open set $U \subset \RR^d$ containing $\overline{\RR^d_+} := [0,+\infty) \times \RR^{d-1}$, and ${\bf a}: U\to \RR^d\setminus\{0\}$ whereas $a_0, \, S: U\to \RR$. We complete the system~\eqref{eq_system}
with the following boundary condition:
\begin{align}
    u(0,\xi) = g(\xi), \quad \forall \xi \in \RR^{d-1},
    \label{boundary}
\end{align}
where  $x \in \RR^d$ is written as $x = (x_1, \xi)$ with $x_1 \in \RR$ and $\xi \in \RR^{d-1}$. We set $\xi = (x_2, \dots, x_d) = (\xi_1, \dots, \xi_{d-1})$.\\

The problem originates from the numerical simulation of landscape evolution models: a landscape is modeled by a function $z: x\mapsto z(x)$ and it is eroded by a fluid layer of height $h$. The sediments created by erosion are transported by the fluid, and the sediment density in the fluid is denoted by $c$. The quantities $(z,h,c)$ are functions of the horizontal coordinate $x$ and time $t$. The evolution system for $(z,h,c)$ introduced in \cite{chen_equations_2014, chen_landscape_2014} and studied in \cite{binard2024well} is written:
\begin{equation}
    \left\{ \begin{array}{lll}
         & \displaystyle
        \dt h+\dive(h {\bf v}) = r, \vspace{1mm} \\
        &\displaystyle
        \dt (hc)+\dive (ch {\bf v}) = \rho_s \, e \left( \frac{h}{H} \right)^m \left( \frac{|{\bf v}|}{V} \right)^n - \rho_s \, s \, \frac{c}{c_{sat}}, \vspace{1mm} \\
        &\displaystyle
        \dt z = K \Delta z - e \left( \frac{h}{H} \right)^m \left( \frac{|{\bf v}|}{V} \right)^n + s \, \frac{c}{c_{sat}}. 
     \end{array} \right.
     \label{eq_landscape_model}
\end{equation}
Here, $\rho_s$ represents the sediment volumetric mass, $e$ is an erosion rate, $s$ a sediment redeposition rate whereas $r$ is a source term representing precipitation, evaporation or infiltration. The constant $K$ is a diffusion constant that represents the soil creep effect. The creep effect is a large-scale effect of soil diffusion that was introduced in~\cite{davis_convex_1892} and used in~\cite{gilbert_convexity_1909} to explain the convexity of hilltop profiles. The erosion is modeled by a stream incision law, see~\cite{howard_channel_1983}. This system is written in a closed form with the fluid velocity written as ${\bf v}=- V \, \nabla (h+z)$, with $V \geq 0$ a characteristic speed of the fluid, and where $h+z$ is the surface elevation. \\

Usually, the fluid speed $|{\bf v}|$ is much larger than the erosion rate $|\partial_t z|$, see e.g. the experiment in~\cite{guerin_streamwise_2020}. In order to focus on the erosion phenomenon, one has to consider large time scales. In this asymptotic regime, the time variations of the fluid height and sediment concentration become negligible: one is left with the equations:
\begin{equation}\label{div-h}
    \left\{ \begin{array}{lll}
         \displaystyle
         \dive(h{\bf v}) = r,\hspace{2mm}\\ 
         \displaystyle
         \dive(ch {\bf v}) = \rho_s \, e \left( \frac{h}{H} \right)^m \left( \frac{|{\bf v}|}{V} \right)^n - \rho_s \, s \, \frac{c}{c_{sat}},
     \end{array} \right.
\end{equation}
supplemented with the equation for $z$ which remains unchanged.
One problem of interest is the erosion of an inclined plane, which can be carried out experimentally \cite{guerin_streamwise_2020}. From a theoretical point of view, this corresponds to a bottom topography $z(x)=-\tan(\theta)x_1+\tilde z(x)$ with $|\tilde z|\ll 1$. The fluid height is prescribed on one side of the plane:
$h(t,0,\xi) = h_{\textrm {ref}}(t,\xi)$ with $\xi \in \RR$. \\

The topography $z$ being fixed, the first equation of \eqref{div-h} is non-linear and could be solved through a fixed point procedure: with $h^0$ an initial guess, one solves
$$
\dive({\bf a}^nh^{n+1})=r,\quad {\bf a}^n=-V\nabla (h^n+z).
$$
We are thus left with a problem similar to the linear problem \eqref{eq_system}-\eqref{boundary}. By splitting space variables, equation \eqref{eq_system} can be written as 
\begin{equation}\label{eq_system-evol}  
\displaystyle
\frac{\partial}{\partial x_1}({\bf a_1} u)(x_1,\xi) + \sum_{j=1}^{d-1} \frac{\partial}{\partial \xi_j} \left({\bf a}_{j+1} u \right) (x_1,\xi) + a_0(x_1,\xi) u(x_1,\xi) = S(x_1,\xi), \quad \forall x_1>0, \forall \xi \in\RR^{d-1}.
\end{equation}
The boundary condition \eqref{boundary} makes sense only if it is not characteristic: in what follows, we will assume that $a_1(0,\xi)> 0,\; \forall \xi\in\RR^{d-1}$. If we further assume that $a_1>0$, equation \eqref{eq_system-evol} can be seen as a first-order partial differential equation where $x_1$ plays the role of time. 
Equation \eqref{eq_system-evol} can be solved by a classical finite-volume scheme; however, a CFL-type condition should be imposed to preserve the positivity of $u$, which could be restrictive for small slopes. Moreover, it may be difficult to design the mesh of the computational domain in the presence of complex or moving boundaries. In this paper, we consider a Particle method. 

\medskip
Particle methods are mesh-free numerical schemes that are mostly used to solve fluid-dynamical problems. The principle of such methods is to use a Lagrangian point of view. The solution is seen as a sum of Dirac distributions carrying some information (density, velocity). These Dirac distributions can be seen as particles advected by the fluid flow and the characteristics of each particle are evolved in time by using the physical laws. These Dirac masses are then approximated by kernel functions $\zeta^h, h>0$, which then provide a classical and computable approximation of the solution.
Smoothed Particle Hydrodynamics (SPH), Particle In Cell (PIC) methods, and vortex methods are examples of such Particle methods. The SPH method was initially introduced to solve numerically problems in astrophysics. It was first designed by Gingold and Monaghan in~\cite{gingold1977smoothed}, and by Lucy in~\cite{lucy1977numerical}. The numerical resolution of astrophysical problems in three dimensions, by finite difference schemes requires a very large number of grid points, whereas the SPH method reaches the same accuracy with a comparatively much smaller number of particles. The SPH method is now used in a wider range of applications like oceanography or structural mechanics and high velocity impacts with large deformations or changes of topology.  
In plasma physics, the Vlasov-Poisson or Vlasov-Maxwell  equations are set in a $6$ dimensional space and it would be very costly to solve them with a classical scheme. The PIC method \cite{evans1957particle} is also a Particle method which also decreases the computational cost. 
The vortex method is a type of Particle method dedicated to the simulation of incompressible Euler equations~\cite{rosenhead1931formation}: it is based on an alternative formulation of the equations involving the vorticity of the fluid velocity, the fluid velocity being reconstructed with theory of potential. This method was designed and implemented in the 70s in two dimensions~\cite{chorin1973discretization}, and later in three dimensions~\cite{leonard1985computing}.
The theoretical analysis of the Particle methods  was initiated by Raviart and coworkers. In \cite{mas1987}, Mas-Gallic and Raviart analyzed a Particle method for one class of first order linear hyperbolic systems. They proved the convergence of the continuous in time scheme in $L^2$. There is now an important literature on the convergence of the SPH method for various types of equations, both in the linear and in the non-linear case: see e.g. ~\cite{ben2000} for non-linear conservation laws, more recent references \cite{evers2018continuum} for a systematical derivation and convergence analysis, and \cite{lind2020review} for a review of recent results and challenges.

\bigskip
Equation ~\eqref{eq_system} is not, in the strict sense of the term, an evolution equation, and classical Particle methods do not apply to it directly.
However, the $x_1$ variable could play the role of a time variable, whereas the boundary condition \eqref{boundary} could be seen as an initial condition provided that it not characteristic, which means that {\bf $a_1\neq 0$} on the boundary $\{0\}\times\RR^{d-1}$.
In this paper we introduce a modified Particle method to solve the stationary equation~\eqref{eq_system}, together with the boundary condition \eqref{boundary}. This method is based on non-stationary Particle methods, and we prove the convergence of the numerical scheme under appropriate assumptions. The idea of the method and a part of the proof of convergence are inspired from \cite{raviart2006analysis}, which analyses a Particle method for linear hyperbolic equations of the first order. The result can be loosely formulated as follows:
\begin{theorem}
    Let $d\in\mathbb{N}$ such that $d\geq 2$. 
    Assume the data ${\bf a}$, $a_0,S,g$ are smooth enough and the characteristic curves fullfill the domain $\RR^d_+$. Then, for all compact sets $X \subset \RR^d_+$, there exist constants $C_X^{(0)}, \dots, C_X^{(3)} >0$ which do not depend on $h$, $\varepsilon$, $\Delta s$, such that, if $C_X^{(0)}\,\Delta s  \leq h,$
    the approximation $\tilde u^{\varepsilon,\Delta s}_h$ given by the scheme \eqref{eq_approx_u_h_epsilon_ds} verifies the error estimate:
    \begin{align}
        \underset{x \in X}{\sup} |u(x) - \tilde u^{\varepsilon,\Delta s}_h(x)| 
        \leq & C_X^{(1)} h^2 + C_X^{(2)} \frac{\varepsilon^{d}}{h^{d}} + C_X^{(3)} \frac{\Delta s}{h},
        \label{eq_error_scheme_intro}
    \end{align}
    where $h$ is the size of support of the kernel function $\zeta^h$, and where $\varepsilon$ and $\Delta s$ are  respectively, the spatial mesh sizes for the transverse variable $\xi\in \RR^{d-1}$ and $s\in\RR_+$.
    \label{th_cv_scheme_intro}
\end{theorem}
\noindent The first term in the right hand side of \eqref{eq_error_scheme_intro} is an interpolation error, stemming from the approximation of $u$ by $u\ast \zeta^h$ with $\zeta^h$ a kernel function approximating the Dirac mass. The function $u\ast\zeta^h$ is written in a integral form involving the characteristic curves and the values of the solution along these curves. The second term is a quadrature error, caused by the discretization of the former integral along characteristics. The last term is the discretization error ``in time'', associated to the time discretization by an Euler scheme of the ordinary differential equations for the characteristic curves and the values of the solution $u$ along these curves.

\medskip
This paper is organized as follow. First, in section~\ref{sec_formula_u}, we introduce a representation formula for the solution $u$ of equation~\eqref{eq_system}-\eqref{boundary} along the characteristic curves. Then, in section~\ref{sec_discretization}, we introduce our numerical method and prove a convergence result on the scheme under suitable assumptions. This result is completed in section ~\ref{sec_test_num} by numerical convergence tests for various test cases. Finally, we apply our numerical scheme to carry out numerical simulations of a landscape evolution model \eqref{eq_landscape_model}. Here the landscape is a tilted plane eroded by a fluid.

\section{A representation formula for the solution using characteristics} \label{sec_formula_u}

In this section, we provide an integral representation of the solution $u$ of \eqref{eq_system}-\eqref{boundary} along the characteristic curves. We shall use this representation of the solution $u$ to build our numerical scheme. We suppose that the boundary $\{0\}\times\RR^{d-1}$ is non characteristic and that ${\bf a}$ is bounded on $\{0\}\times\RR^{d-1}$: as a consequence, there exist two positive constants $\alpha$ and $\beta$ such that
\begin{align}
    & \forall \xi \in \RR^{d-1} , \quad 0 < \alpha \leq { \bf a}_1(0,\xi), \quad |{ \bf a}(0,\xi)| \leq \beta. \label{eq_a_incoming}
\end{align}
The first assumption is consistent with the original problem \cite{binard2024well} where water is injected in the domain from the boundary $\{0\}\times\RR^{d-1}$ of $\RR^d_+$. We further assume that the function ${\bf a} \in C^1(U)$. The Cauchy-Lipschitz theorem applies, and we can define the characteristic curves $x_\xi$ as the maximal solutions of 
\begin{equation}
    x_\xi'(s) = {\bf a}(x_\xi(s)), \text{ with} \quad x_\xi(0) = (0,\xi) \in \{0\} \times \RR^{d-1}.
    \label{vector_field}
\end{equation}
The characteristics $x_\xi$ are defined on the interval $(T_{\min}(\xi),  T_{\max}(\xi))$ with $T_{\min}(\xi)<0<T_{\max}(\xi)$ for all $\xi\in\RR^{d-1}$.
We associate to the characteristic curves $x_\xi, \xi\in\RR^{d-1}$ the flow function $\Phi$ defined on  
$$\displaystyle\Omega=\{ (s,\xi) \, ; \xi \in \RR^{d-1} , \, s \in (T_{\min}(\xi),T_{\max}(\xi)) \, \}$$
as
\begin{equation}\label{def_flow}
\begin{array}{ll}
\displaystyle
\Phi : \Omega & \to \RR^d,\\
(s,\xi) & \mapsto \Phi(s,\xi) = x_\xi(s).
\end{array}
\end{equation}

\begin{figure}[ht]
    \center
     \begin{tikzpicture}[scale = 1]
     \node (0,0) {\includegraphics[width=0.8\textwidth]{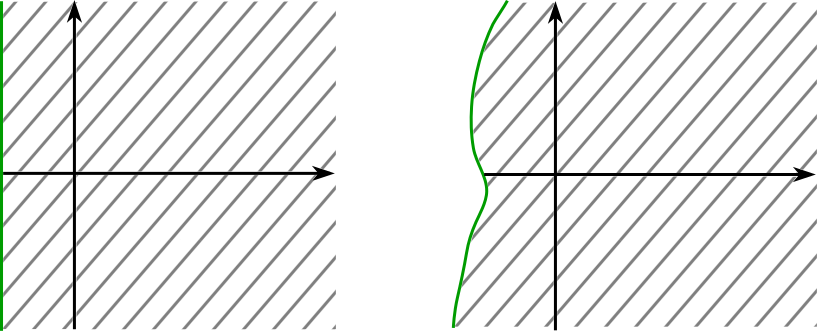}};

     \draw[black] (-5.5,3) node {$\xi$};
     \draw[black] (-1.05,-0.25) node {$s$};
     \draw[ForestGreen] (-6.85,3) node {$s=-\ell$};
     
     \draw[black] (2.4,3) node {$\xi$};
     \draw[black] (6.9,-0.25) node {$s$};
     
     \draw[red] (0,0) node {\Large $\longrightarrow$};
     \draw[red] (0,-0.3) node {\Large $\Phi$};
     \end{tikzpicture}
     
     \caption{Domains $\Omega_\ell\subset\Omega$ and $\Phi(\Omega_\ell) = \widetilde{U}_\ell\subset U$.}
     \label{fig_flow_phi}
\end{figure}
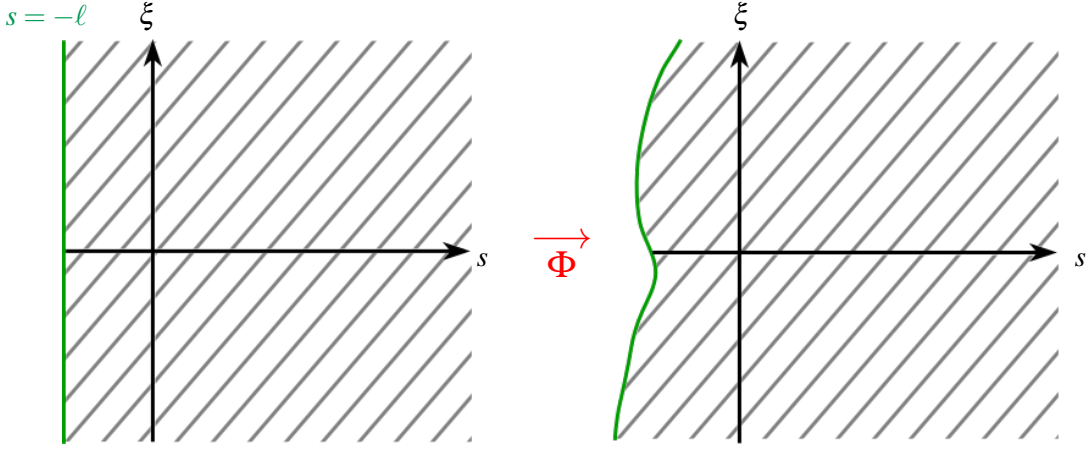

\noindent
We then have the following representation formula for the solution $u$ of System \eqref{eq_system}-\eqref{boundary} along the characteristics.

\begin{lemma}\label{def-u-char}
Assume ${\bf a} \in C^2(U)$, $\, a_0, \, S \in C^1(U)$; then, the solution $u\in C^1(\Phi(\Omega))$ of \eqref{eq_system}-\eqref{boundary} satisfies
\begin{align}
        \forall (s,\xi)\in \Omega, \quad &u(\Phi(s,\xi)) = \left[ g(\xi) + 
        \int_0^s S(\Phi(t,\xi)) \exp \left(I(t,\xi) \right) dt \right] \exp \left(- I(s,\xi) \right), \label{eq_expr_sol_u}\\
        &\text{where} \quad I(s,\xi) = \int_0^s \big( \dive \left({ \bf a }(\Phi(\tau,\xi)) \right) + a_0(\Phi(\tau,\xi)) \big) d\tau. \nonumber
    \end{align}
\end{lemma}

\noindent 
We want to solve the system \eqref{eq_system}-\eqref{boundary} on the domain $\RR^d_+$. The assumption \eqref{eq_a_incoming} on ${\bf a_1}$ being satisfied, this domain is positively invariant under the flow of the differential system \eqref{vector_field}: 
$$
\displaystyle
\Phi(s,\xi)\in\RR^d_+,\quad \forall s\geq 0,\quad \forall \xi\in\RR^{d-1}.
$$

\noindent
However, this does not ensure that we can define the solution $u$ on the whole domain $\RR^d_+$ and we have to make further assumptions on the flow $\Phi$: 
\medskip
\begin{assumption}~
    \begin{enumerate}[(i)]
        \item There exists $\ell>0$ such that the flow $\Phi$ is defined on a domain $\Omega_\ell=(-\ell, +\infty)\times\RR^{d-1}\subset\Omega$ and $\Phi: \Omega_\ell \to \Phi(\Omega_\ell) := \widetilde{U}_\ell \subset U$ is a $C^1$-diffeomorphism (see figure \ref{fig_flow_phi}):
        $$\displaystyle \forall x \in \widetilde{U}_\ell,\quad \exists (s,\xi)\in(-\ell,+\infty)\times\RR^{d-1},\quad x=x_\xi(s).$$
        \item There exists $\delta>0$ such that $U_\delta = (-\delta,+\infty) \times \RR^{d-1}\subset \widetilde{U}_\ell\subset U$. This is illustrated in Figure~\ref{fig_caract_curves}. 
    \end{enumerate}
    \label{assump_flow}
\end{assumption}
These assumptions ensure that we can define the solution $u$ on $\RR^d_+$. However, note that we have required a stronger assumption so as to define the solution $u$ on the domain $U_\delta$. Indeed, in the interpolation procedure, we approximate $u$ by $u\ast\zeta^h$ and we need that the solution $u$ to be defined on a domain $U_h$ to compute $u\ast\zeta^h$ on $\RR^d_+$. Here, we can carry out these computations by assuming $0<h<\delta$. 

\medskip
Note also that our setting does not deal with the problem of ``dry areas'' when the characteristics curves starting from the boundary $\{0\}\times\RR^{d-1}$ do not cover entirely the domain $\RR^{d-1}_+$. We will discuss in the numerical part some strategies to extend our framework.\\ 
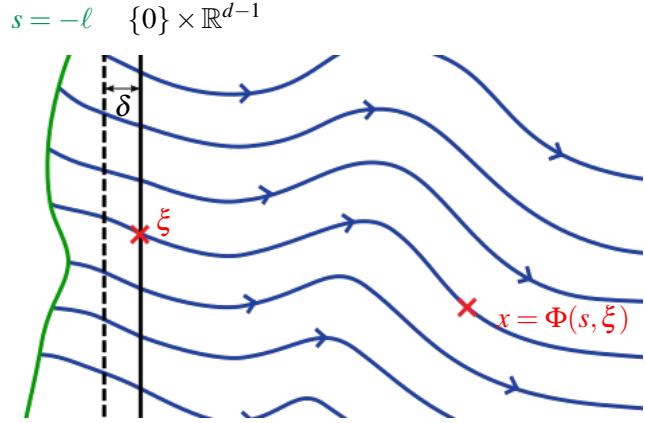
\begin{figure}[ht]
    \center
     \begin{tikzpicture}[scale = 1]
     \node (0,0) {\includegraphics[width=0.5\textwidth]{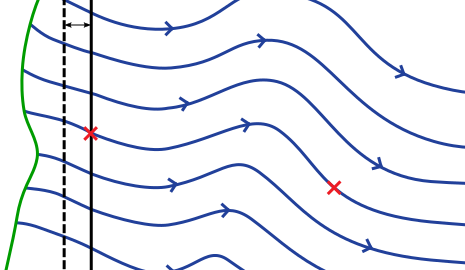}};
     \draw[red] (-2.2,0.2) node {$\xi$};
     \draw[red] (3.1,-1.1) node {$x=\Phi(s,\xi)$};
     
     \draw[black] (-1.8,2.9) node {$\{0\} \times \RR^{d-1}$};
     \draw[black] (-2.75,1.75) node {$\delta$};
     
     \draw[ForestGreen] (-3.7,2.9) node {$s=-\ell$};
     
     \draw[white] (5,0) node { };
     \end{tikzpicture}
     
     \caption{Characteristic curves on the domain $\widetilde {U}_\ell$ and definition of $U_\delta$.}
     \label{fig_caract_curves}
\end{figure}


For the reader's convenience, we list  the various domains considered in the paper:
\begin{itemize}
    \item $U$ is the domain of definition of ${\bf a}$.
    \item $\Omega$ is the domain of definition of $\Phi$, and $\Phi(\Omega) \subset U$.
    \item $\Omega_\ell = (-\ell,+\infty) \times \RR^{d-1} \subset \Omega$.
    \item $\widetilde{U}_\ell = \Phi(\Omega_\ell) \subset U$.
    \item $U_\delta = (-\delta,+\infty) \times \RR^{d-1} \subset \widetilde{U}_\ell$.
\end{itemize}

Under suitable assumptions on the vector field ${\bf a}$, the flow $\Phi$ satisfies the assumptions formulated above. 
The following proposition is proved in Appendix~\ref{appendix_proof_prop_assump}.

\begin{proposition}
   Let $L>0$, and assume that $U \supset U_L = (-L,+\infty) \times \RR^{d-1}$. Suppose that the function ${\bf a}$ is globally Lipschitz continuous on $U$ and satisfies \eqref{eq_a_incoming}. Then, 
   \begin{enumerate}[(i)]
       \item For all $\xi \in \RR^{d-1}$, equation~\eqref{vector_field} admits a unique maximal solution $x_\xi$. This solution is such that: $T_{\min}(\xi) < 0$, $T_{\max}(\xi) = +\infty$. The solution $x_\xi$ is $C^1((T_{\min}(\xi),+\infty))$ and the flow $\Phi$ is $C^1(\Omega)$. 
       \item For all $\xi \neq \eta \in \RR^{d-1}$, the curves $s \mapsto x_{\xi}(s)$ and $s \mapsto x_{\eta}(s)$  do not intersect each other.
       \item There exists $\ell >0$ such that $\Phi$ verifies the first statement of assumption~\ref{assump_flow}.
       \item Assume that ${\bf a}_1(x) \geq \alpha >0,\:\forall x \in U$. Then $\Phi$ verifies the second statement of assumption \ref{assump_flow}.
   \end{enumerate}
\label{prop_ex_flow}
\end{proposition}

We also consider the smoothness of the solution $u$. In order to derive error estimates for our numerical scheme, we need additional regularity on the solution $u$ which is obtained by assuming more regularity on ${\bf a}$, $a_0, \, S$ and $g$.

\begin{proposition}
    Let $m \in \NN$, and suppose that 
    ${\bf a}\in C^{m+1}(U)$, $g \in C^{m}(\RR^{d-1})$ whereas $a_0,~ S \in C^{m}(U)$. Then, one has: 
    \begin{enumerate}[(i)]
        \item $\Phi: \Omega_\ell \to \widetilde {U}_\ell\subset U$ is a $C^{m+1}$ diffeomorphism.
        \item The system \eqref{eq_system}-\eqref{boundary} is well-posed on $\widetilde{U}_\ell$ and 
        the solution $u$ belongs to $C^{m}(\widetilde{U}_\ell)$.
    \end{enumerate}
    
    \label{prop_caract_domain}
\end{proposition}
\noindent See~\cite[chap V]{dieudonne1969elements} for the proof of the first item of Proposition~\ref{prop_caract_domain}.

\section{The discretization method} \label{sec_discretization}

In this section we describe the Particle method and provide a proof of convergence of the numerical scheme. We follow the strategy of the SPH method: the solution is written as a sum of Dirac distributions. The Dirac masses are advected by the flow described by the characteristic equations whereas the values of the solution at the Dirac masses locations are evolved by using the partial differential equation. In order to deal with functions, the Dirac masses are replaced by a compactly supported family of functions $\xi^h, h>0$ (so called kernel functions) approximating the Dirac distribution. In our setting, this program can be described as follow: we first replace the solution $u$ by its regularized version $u\ast\zeta^h$. This later function can be written as an integral over $\tilde U_\ell$ for $h$ sufficiently small. We then use the flow $\Phi$ to perform a change of variable where the functions are evaluated along characteristics and the integral is taken on $\Omega_\ell$. We use a quadrature formula to discretize this integral with a spatial step size $\varepsilon>0$ in the transverse direction $\xi\in\RR^{d-1}$ and $\Delta s>0$ in the direction $s>-\ell$. We obtain a fully discretized scheme by using the characteristic equations to compute an approximation of the particles position and an approximation of the solution at these discretization points.

In section~\ref{subsection_1}, we introduce the  kernel function $\zeta^h$, defined as an approximation of the Dirac delta distribution, and the change of variables formula in integrals defined on $\Phi(\Omega_\ell)$ associated to the characteristic flow $\Phi$. This change of variable is used to write the integral defining $u\ast\zeta^h$ along the characteristic curves. In section~\ref{subsection_2}, we present the quadrature formula, that consists of a semi-discretization of the former integral in the direction $\xi\in\RR^{d-1}$. In section~\ref{subsection_3}, we introduce the discretization of the characteristic curves which in turn introduces our fully discrete scheme. Finally, the convergence theorem is proved in section~\ref{subsection_4}.

\subsection{The smoothed particle formalism} \label{subsection_1}

In what follows, we assume that $x \in \RR^d_+ := (0,+\infty) \times \RR^{d-1}$ and we want to approximate $u$ by a smoother function $u\ast\zeta^h$ on this domain. For that purpose, we define a kernel function $\zeta^h$ , an approximation of the Dirac delta distribution:
\begin{equation}
    \forall x \in \RR^{d}, \quad \zeta^h(x) = \frac{1}{h^{d}} \zeta \left(\frac{|x|}{h} \right)
    \label{def_zetah}
\end{equation}
with $\zeta : [0,+\infty) \to \RR^+$ a function compactly supported in $[0,1]$ which verifies
\begin{equation}
	\int_{\RR^d} \zeta \left(|x| \right) dx = \text{meas} (\mathbb{S}^{d-1}) \int_0^1 \zeta(t) t^{d-1} dt = 1,
	\label{def_zeta}
\end{equation}
where $\text{meas} (\mathbb{S}^{d-1})$ is the $d-1$ dimensional measure of the $d-1$ dimensional sphere.
The support of $\zeta^h$ is included into $\overline{B}_d(0,h)$ where, for all $d \in \NN$, $x\in\RR^{d}$ and $h>0$, the sets $B_d(x,h),\overline{B}_d(x,h)$ are given by
\begin{align}
    &B_d(x,h) := \{y \in \RR^d \, ; \; |x-y| < h \},\\
    &\overline{B}_d(x,h) := \{y \in \RR^d \, ; \; |x-y| \leq h \}.
    \label{def_ball}
\end{align}

\noindent Then, for any $f \in C^m(U)$ and $x \in \RR^{d}_+$, the function $y \mapsto f(y) \zeta^h(x-y)$ is compactly supported in $U_\delta$ provided that $0<h<\delta$. This function can be extended by $0$ to the domain $\widetilde{U}_\ell$ and we get the approximation formula:
\begin{align}
    \displaystyle f(x) \approx f\ast\zeta^h(x)= \int_{U_\delta} f(y) \zeta^h(x-y) dy = \int_{\widetilde{U}_\ell} f(y) \zeta^h(x-y) dy.
    \label{eq_approx_f(x)}
\end{align}

Next, we use the fact that the flow $\Phi$ is a $C^1$ diffeomorphism so that we can perform a change of variables in the formula \eqref{eq_approx_f(x)}. Let us denote by $J_\Phi$ the Jacobian matrix of $\Phi$, given by:
 \begin{equation}
     \displaystyle J_\Phi(s,\xi) = \left[ \begin{array}{cccc}
        \displaystyle \frac{\partial \Phi_1}{\partial s} & \displaystyle \frac{\partial \Phi_1}{\partial \xi_1} & \displaystyle \dots & \displaystyle \frac{\partial \Phi_1}{\partial \xi_{d-1}} \\
        \displaystyle \vdots & \displaystyle \vdots & & \displaystyle \vdots\\
        \displaystyle \frac{\partial \Phi_d}{\partial s} & \displaystyle \frac{\partial \Phi_{d}}{\partial \xi_1} & \displaystyle \dots & \displaystyle \frac{\partial \Phi_d}{\partial \xi_{d-1}}
     \end{array} \right]
     \label{eq_jac_phi}
 \end{equation}
As $\Phi$ is invertible, $\det J_\Phi$ does not vanish. Moreover, from the relation $\Phi(0,\xi)=(0,\xi)$, one deduces that
$$
\displaystyle
\nabla_\xi\Phi_1(0,\xi)=0,\quad \left. \frac{\partial (\Phi_2, \dots, \Phi_d)}{\partial (\xi_1, \dots, \xi_{d-1})} \right|_{(0,\xi)}={\rm Id}_{\RR^{d-1}},\quad\forall\xi\in\RR^{d-1}.
$$
As a result, for all $\xi \in \RR^{d-1}$, one has 
$$
\displaystyle
(\det J_\Phi) (0,\xi) = \partial_s \Phi_1(0,\xi) ={ \bf a}_1(0,\xi) \geq \alpha > 0.
$$ 
Therefore, by using the Liouville formula, one deduces that
$$\forall s \geq -\ell, \, \forall \xi \in \RR^{d-1}, \; (\det J_\Phi) (s,\xi)=\exp\left(\int_0^s\dive({\bf a})(\Phi(t,\xi))dt\right)(\det J_\Phi) (0,\xi) > 0.$$
Then, for any $f \in C^0(\widetilde{U}_\ell)$, by performing the change of variables $\Phi:(-\ell,+\infty)\times\RR^{d-1}\to\tilde U_\ell$, we get:
\begin{align}
    \displaystyle \int_{\widetilde{U}_\ell} f(x) dx = \int_{-\ell}^{+ \infty} \int_{\RR^{d-1}} f(\Phi(s,\xi)) (\det J_\Phi)(s,\xi) d\xi ds.
    \label{eq_change_var}
\end{align}
Consequently, an integral on $\widetilde{U}_\ell$ can be converted into an integral on $(-\ell, +\infty)\times\RR^{d-1}$,
provided that one can determine the function along the characteristic curves. It is easier to perform computations on this domain, and to create a mesh on it.\\

We can now apply the approximation \eqref{eq_approx_f(x)} to the solution~$u$ of equation~\eqref{eq_system} together with the change of variables \eqref{eq_change_var} since we have a representation formula for the solution $u$ along characteristic curves. One finds:
\begin{equation}
\displaystyle
    \forall x \in \RR^d_+, 
    \quad u(x) \approx \int_{\widetilde{U}_\ell} u(y) \zeta^h(x-y) dy 
     = \int_{-\ell}^{+ \infty} \int_{\RR^{d-1}} u(\Phi(s,\xi)) \zeta^h(x-\Phi(s,\xi)) (\det J_\Phi) (s,\xi) d\xi ds.
     \label{eq_approx_u_int}
\end{equation}

\noindent We conclude this section with an error estimate on $u - u \ast \zeta^h$.
\begin{lemma}\label{diff-u-uh}
Let $u \in C^2(\widetilde{U}_\ell)$, $\delta>0$ such that $\overline{U_\delta} \subset \widetilde{U}_\ell$, and $x \in \RR^d_+$. Then there exists $C(x,\delta)>0$ such that, for any $0<h<\delta$, one has
\begin{equation}
    \displaystyle
    |u(x) - u \ast \zeta^h(x)| \leq C(x,\delta) h^2.
    \label{eq_lemma_bound_convol}
\end{equation}
\end{lemma}

\begin{remark}
This approximation formula is order two, but this could be increased if we remove the hypothesis of positivity on $\zeta$. Actually, if the $m$ first moments of $\zeta$ vanish, then the approximation has order $m$, see \cite{raviart2006analysis}. However, for the applications we have in mind, like the simulation of a landscape evolution model, the unknown $u$ is either a fluid height or a sediment concentration, so we restrict our discussion to $\zeta$ a non-negative function. 
\end{remark}

\begin{proof}
Let $x\in\RR^d_+$: we expand the function $u$  up order $2$ with respect to $z=y-x\to 0$. As $u$ is $C^2(\widetilde{U}_\ell)$, for all $y \in \overline{U_\delta}$, one has: 
    \begin{align*}
        u(x)-u(y) = (x-y) \cdot \nabla u(x) + R(x,y),
    \end{align*}
    where $R$ is a remainder which verifies: 
    $$|R(x,y)| \leq |x-y|^2 \underset{z \in [x,y]}{\sup} |H(u)(z)|,$$
    Where $|H(u)(z)|$ denotes an appropriate matrix norm of the Hessian matrix of $u$ at the point $z$. Thus, as the integral of $\zeta^h$ is equal to $1$, one finds
    
    \begin{align*}
        u(x) - \int_{\RR^d} u(y) \zeta^h(x-y) dy &= \int_{\RR^d} \left(u(x)-u(y) \right) \zeta^h(x-y) dy\\
        & = \int_{\RR^d} (x-y) \zeta^h(x-y) dy \cdot \nabla u(x) + \int_{\RR^d} R(x,y) \zeta^h(x-y) dy.
    \end{align*}
    
    As $\zeta^h$ is a radial function, we have
    $$\displaystyle\int_{\RR^d} (x-y) \zeta^h(x-y) dy = 0.$$
    Moreover, the support of the function $y \mapsto \zeta^h(x-y)$ is contained in $\overline{B}_d(x,h) \subset \overline{B}_d(x,\delta) \subset \widetilde{U_\ell}$. Consequently, we have
    
    $$
        |u(x) - u \ast \zeta^h(x)| \leq \sup_{y\in \overline{B}_d(x,h)} |R(x,y)| \leq C(d) h^2 \, \underset{\substack{y \in \overline{B}_d(x,\delta)\\1 \leq i,j \leq d}}{\sup } \, |\partial^2_{ij} u(y)| = C(x,\delta) h^2,
    $$
    where $C(d)$ is a constant which only depends on of the dimension $d$ and we can set
    \begin{align} 
        C(x,\delta)=C(d)\underset{\substack{y \in \overline{B}_d(x,\delta)\\
        1 \leq i,j \leq d}}{\sup } \, |\partial^2_{i,j} u(y)| < +\infty,
        \label{eq_lemma_bound_convol_cst}
    \end{align}
    because $\overline{B}_d(x,\delta)$ is compact.
    This completes the proof of the lemma.
\end{proof}

\begin{remark}
To obtain the bound~\eqref{eq_lemma_bound_convol} on $\RR_+^d$, we need to suppose that $u$ is defined and smooth on the larger domain $U_\delta$ (see formula~\eqref{eq_lemma_bound_convol_cst} for the bound $C(x,\delta)$). Note that with the definition of $C(x,\delta)$, it is a straightforward consequence that for any $K$ compact subset of $\RR^d_+$, there exists $C(X,\delta)$ such that we have the uniform estimate:
$$
\displaystyle
\max_{x\in K}|u(x)-u\ast\zeta^h(x)|\leq C(X,\delta)h^2.
$$
\end{remark}

\subsection{Quadrature formula and the SPH-like discretization} \label{subsection_2}

In this section we introduce a discretization of the integral approximation  \eqref{eq_approx_u_int} of $u$. For that purpose, we first recall a theorem from~\cite{raviart2006analysis}, which provides a quadrature formula and an error estimate for functions in Sobolev spaces. 
\begin{theorem}
    Let $m$ be an integer such that $m>d$. Then there exists $C>0$ such that, $\forall f \in W^{m,1}(\RR^{d})$ and $\forall \varepsilon>0$ :
    \begin{align}
       \left|\int_{\RR^{d}} f(x) dx - \varepsilon^{d} \sum\limits_{k \in \ZZ^{d}} f(x_k)\right| \leq C \varepsilon^m \| f \|_{W^{m,1}(\RR^{d})},
    \end{align}
    \label{th_quadrature_Raviard}
    where $(x_k)_i= k_i \, \varepsilon,\quad \forall i=1,\dots,d$.
\end{theorem}

\begin{remark}
The condition $m>d$ is used to have the continuous embedding $W^{m,1}(\RR^{d}) \hookrightarrow C^0(\RR^{d})$, which is needed to give a meaning to the point values $f(x_k)$.
\end{remark}

\noindent
Now we discretize the integral \eqref{eq_approx_u_int} with respect to $\xi\in\RR^{d-1}$; this yields the following semi-discrete approximation of $u$:
\begin{equation}\label{u-int-xi}
    \displaystyle
    u(x) \approx \varepsilon^{d-1}\sum_{k\in\mathbb{Z}^{d-1}} \int_{-\ell}^\infty u(\Phi(s,\xi_k)) \zeta^h(x-\Phi(s,\xi_k)) ({\rm det}J_\Phi)(s,\xi_k) ds,
\end{equation}
with $(\xi_k)_i = k_i \, \varepsilon, \quad \forall i=1,\dots,d-1$. Introducing $x_k(s) = \Phi(s,\xi_k)$, $\omega_k(s) = \varepsilon^{d-1}({\rm det}J_\Phi)(s,\xi_k)$ and $u_k(s) = u(x_k(s))$, the approximation \eqref{u-int-xi} reads
$$
\displaystyle
u(x)\approx \sum_{k\in\mathbb{Z}^{d-1}}\int_{-\ell}^\infty \omega_k(s) u_k(s) \zeta^h(x-x_k(s))ds,
$$
in the spirit of the SPH method. \\

Now we derive evolution equations for the quantities $x_k,u_k$ and $\omega_k$: this will provide a natural way to compute these quantities on a mesh discretizing the interval $(-\ell, +\infty)$. 
\begin{proposition}
    Let $u$ be the solution of equation~\eqref{eq_system}, with boundary condition \eqref{boundary}. Let $k \in \ZZ^{d-1}$. Then, the quantities $x_k(s)=\Phi(s,\xi_k)$, $u_k(s)= u(x_k(s))$ and $\omega_k(s)= \varepsilon^{d-1} (\det J_\Phi) (s,\xi_k)$ satisfy the differential system
	\begin{subequations}
	    \begin{empheq}[left=\empheqlbrace]{align}
            & x_k'(s) = {\bf a}(x_k(s)), \label{syst_part1}\\
            & \omega_k'(s) = (\dive {\bf a})(x_k(s)) \omega_k(s),\quad\forall s>-\ell \label{syst_part2}\\
            & u_k'(s) = -u_k(s) \left[ (\dive {\bf a}) (x_k(s)) + a_0(x_k(s)) \right] + S(x_k(s)), \label{syst_part3bis} \\
	        & x_k(0) = (0,\xi_k), \quad \omega_k(0) = \varepsilon^{d-1}{ \bf a}_1(0,\xi_k), \quad u_k(0) = g(\xi_k). \label{syst_part4}
	    \end{empheq}
	    \label{syst_part}
	\end{subequations}
	\label{prop_syst_u_k(s)}
\end{proposition}

\begin{proof}
Equation~\eqref{syst_part1} follows  from the definition of the characteristic system \eqref{vector_field}. Moreover, equation \eqref{syst_part2} is a direct consequence of Liouville's formula:
\begin{align*}
    \frac{d}{ds} (\det J_\Phi)(s,\xi) = (\dive {\bf a})(\Phi(s,\xi)) (\det J_\Phi)(s,\xi).
\end{align*}

Next, we prove \eqref{syst_part3bis}. By applying the chain rule, equation~\eqref{eq_system} and the characteristic system \eqref{vector_field},
one obtains:
\begin{align*}
    \frac{d}{ds} \left( u(x_k(s)) \right) &= x_k'(s). \nabla u(x_k(s)) 
    = {\bf a}(x_k(s)). \nabla u(x_k(s))\\
    &= -u(x_k(s)) \left[ (\dive {\bf a})(x_k(s))+a_0(x_k(s)) \right] + S(x_k(s)).
\end{align*}
This completes the proof of the proposition.
\end{proof}

\noindent
We also consider the additional quantity $\rho_k=\omega_ku_k$. It is an easy consequence of proposition \eqref{prop_syst_u_k(s)} that these functions satisfy the following result.
\begin{corollary}
    $\forall k \in \ZZ^{d-1}$, the function $\rho_k=\omega_ku_k$ is the solution of the Cauchy problem:
	\begin{subequations}
	    \begin{empheq}[left=\empheqlbrace]{align}
            & \rho_k'(s) + a_0(x_k(s)) \rho_k(s) = \omega_k(s) S(x_k(s)),\quad \forall s>-\ell, \label{syst_part3}\\
	        & \rho_k(0) = \varepsilon^{d-1}{ \bf a}_1(0,\xi_k) g(\xi_k). \label{syst_part3_0}
	    \end{empheq}
	    \label{syst_part_rho}
	\end{subequations}
	\label{corr_syst_rho_k(s)}
\end{corollary}


Systems \eqref{syst_part} and \eqref{syst_part_rho} are the analogue of the equations in the SPH method.
The difference here is that we replace the time variable by a variable $s$ that parameterizes the space coordinate $x_1=\Phi_1(s,\xi)>0$.\\

We end this section by deriving a priori estimates on the unknowns $\rho_k$ and $\omega_k$ that will be useful in the proof of convergence of our numerical scheme. For that purpose, we make the following additional hypothesis:
\begin{assumption}
  The data ${\bf a}$, and $g,a_0,S$ satisfy the regularity assumptions.
    $$\displaystyle
    {\bf a} \in C^1(\widetilde{U}_\ell) \cap W^{2,\infty}(\widetilde{U}_\ell),\quad g \in L^\infty(\RR^{d-1}), \quad a_0, \, S \in C^1(\widetilde{U}_\ell) \cap W^{1,\infty}(\widetilde{U}_\ell).$$
    \label{assump_reg}
\end{assumption}
\noindent
We obtain the following estimates on the weight $\omega_k$ and on the functions $\rho_k$.
\begin{lemma}
Suppose that ${\bf a}$ and $g,a_0,S$ satisfy the assumption \ref{assump_reg}.
Let $(t,T) \in \RR^2$ such that $-\ell < t < 0 < T < +\infty$. Then there exist constants $C_0(t,T)$, $C_1(t,T)$, $C_2(t,T)$, $C_3^0(t,T)$, $C_3^2(t,T) >0$ independent of $\varepsilon$ and $h$ such that for all $k \in \ZZ^{d-1}$ and for all $s\in[t,T]$, one has:
\begin{align}
    &C_0(t,T) \varepsilon^{d-1} \leq \omega_k(s) \leq C_1(t,T) \varepsilon^{d-1}, \label{eq_bound_w}\\
    &|\omega_k''(s)| \leq C_2(t,T) \varepsilon^{d-1}, \label{eq_bound_w''}\\
    &|\rho_k(s)| \leq C_3^0(t,T)  \varepsilon^{d-1},
    \quad |\rho_k''(s)| \leq C_3^2(t,T)  \varepsilon^{d-1}. \label{eq_bound_rho}
\end{align}
\label{lemma_bound_wk_rhok}
\end{lemma}

\begin{proof}
Let us fix $k \in \ZZ^{d-1}$ and $s\in[t,T]$. Then, by using assumption~\eqref{eq_a_incoming}, one finds
\begin{align*}
    \alpha \varepsilon^{d-1} \leq \omega_k(0)  = \varepsilon^{d-1}{ \bf a}_1(0,\xi_k)
    &\leq \beta \varepsilon^{d-1}.
\end{align*}
A direct integration of equation \eqref{syst_part2} yields:
\begin{align*}
    \omega_k(s)  = \omega_k(0) \exp \left(\int_0^s (\dive {\bf a})(x_k(t)) dt \right)\leq  \beta\varepsilon^{d-1}  \exp \left((T-t)\|\dive {\bf a}\|_{L^\infty(\widetilde{U}_\ell)} \right) 
    := C_1(t,T)\varepsilon^{d-1}.
\end{align*}
Similarly, one has
\begin{align}
    \omega_k(s) & \geq \alpha \varepsilon^{d-1} \exp \left( (t-T)\|\dive {\bf a}\|_{L^\infty(\widetilde{U}_\ell)} \right) 
    := C_0(t,T) \varepsilon^{d-1}.
    \label{eq_bound_wk_proof}
\end{align}\\
    

From equations~\eqref{syst_part2} and~\eqref{eq_bound_wk_proof}, one deduces that
    \begin{align*}
        |\omega_k''(s)| 
        & = |(\dive {\bf a})(x_k(s))^2 \omega_k(s) + {\bf a}(x_k(s)). \nabla (\dive {\bf a})(x_k(s)) \omega_k(s)|\\
        & \leq 2\|a\|^2_{W^{2,\infty}(\widetilde{U}_\ell)^d}  \omega_k(s)
        \leq 2\|a\|^2_{W^{2,\infty}(\widetilde{U}_\ell)^d}C_1(t,T):=C_2(t,T)  \varepsilon^{d-1}.
    \end{align*}\\

Now let us show the first estimate of~\eqref{eq_bound_rho}. By using \eqref{syst_part_rho} and for all $s\in(t,\, T)$, one finds that: 
\begin{align*}
    \left| \rho_k(s) \right| 
    \! = \!& \left| \Bigg[ \varepsilon^{d-1}{ \bf a}_1(0,\xi_k) \, g(\xi_k) \! + \! \!
    \int_0^s \! \omega_k(\mu) S(x_k(\mu))
    \exp \left(\int_0^\mu a_0(x_k(\tau)) d\tau \! \right) d\mu \Bigg]
    \exp \left(\! - \! \int_0^s a_0(x_k(\mu)) d\mu \right) \right|\\
    &\leq \varepsilon^{d-1} \exp\left((T-t)\|a_0\|_{L^\infty(\widetilde{U}_\ell)}\right) \left(\|g\|_{L^\infty(\widetilde{U}_\ell)} \|a\|_{L^\infty(\widetilde{U}_\ell)} + (T-t)C_1(t,T) \|S\|_{L^{\infty}(\widetilde{U}_\ell)} \right) \\
    & := C^0_3(t,T) \varepsilon^{d-1}.
\end{align*}
Then, from \eqref{syst_part_rho} and by using the bounds on $\rho_k$, one obtains:
\begin{eqnarray*}
\displaystyle
    \left| \rho_k'(s) \right| 
    &=& \left|- a_0(x_k(s)) \rho_k(s) + \omega_k(s) S(x_k(s)) \right|\\
    \displaystyle
    && \leq \varepsilon^{d-1}\left(\|a_0\|_{L^\infty(\widetilde{U}_\ell)}C_3^0(t,T)+\|S\|_{L^{\infty}(\widetilde{U}_\ell)}C_1(t,T)\right):=C_3^1(t,T)\varepsilon^{d-1}.
\end{eqnarray*}\\

Finally, by differentiating \eqref{syst_part_rho} with respect to $s$, one deduces that
\begin{eqnarray*}
    \left| \rho_k''(s) \right| &=&  \left| x_k'(s).\nabla a_0(x_k(s)) \rho_k(s) + a_0(x_k(s)) \rho_k'(s) - \omega_k'(s) S(x_k(s)) - \omega_k(s) x_k'(s) . \nabla S(x_k(s)) \right|\\
    &&\leq  C_3^2(t,T)  \varepsilon^{d-1}
\end{eqnarray*}

with
$$
\displaystyle
C_3^2(t,T)=C(d) \left(\|a_0\|_{W^{1,\infty}(\widetilde{U}_\ell)}\left(\|{\bf a}\|_{L^{\infty}(\widetilde{U}_\ell)}C_3^0(t,T)+C_3^1(t,T)\right)+C_1(t,T)\|{\bf a}\|_{W^{1,\infty}(\widetilde{U}_\ell)}\|S\|_{W^{1,\infty}(\widetilde{U}_\ell)}\right).
$$

This completes the proof of the lemma.
\end{proof}

\subsection{Discretization of the particle trajectories: the fully discrete scheme} \label{subsection_3}

Now we introduce the discretization with respect to $s \in (-\ell, +\infty)$. Indeed, we want to approximate $u\ast\zeta^h(x)$ for all $x\in\RR^d_+$. For that purpose, we need to introduce a discretization of $\Omega_\ell=(-\ell, +\infty)\times \RR^{d-1}$
Let us denote $n_\ell\in\mathbb{Z}$ such that $(n_\ell-1) \Delta s< -\ell \leq n_\ell \Delta s$, where $\Delta s>0$. We will approximate $u$ with the formula:
\begin{align}
    \displaystyle
    u(x) & \approx \sum_{k\in\mathbb{Z}^{d-1}}\int_{-\ell}^{+\infty} \omega_k(s) u_k(s) \zeta^h(x-x_k(s)) ds, \nonumber \\
    \displaystyle
    &\approx \Delta s \sum_{i=n_\ell}^{+\infty} \sum_{k\in\mathbb{Z}^{d-1}} \omega_k(s_i) u_k(s_i) \zeta^h(x-x_k(s_i)), \nonumber \\
    &= \sum_{i=n_\ell}^{+\infty} \sum_{k\in\mathbb{Z}^{d-1}} \widetilde{\omega}_k(s_i) u_k(s_i) \zeta^h(x-x_k(s_i)), \nonumber\\
    &= \sum_{i=n_\ell}^{+\infty}\sum_{k\in\mathbb{Z}^{d-1}} \widetilde{\rho}_k(s_i) \zeta^h(x-x_k(s_i)) := \Pi_h^{\varepsilon,\Delta s}(u)(x), \label{eq_approx_u_def}
\end{align}
where $s_i= i \Delta s$ and $\forall i \geq n_\ell$, \, $\tilde{\omega}_k(s_i)=\Delta s \omega_k(s_i)$ and $\tilde{\rho}_k(s_i) = \Delta s\rho_k(s_i)$. Note that $\widetilde{\rho}_k, \widetilde{\omega}_k$ and $\rho_k,\omega_k$ satisfy the same differential system: the only difference lies in the initial conditions that are multiplied by a factor $\Delta s$. In what follows, we remove the tilde from $\widetilde{\rho}_k$ and $\widetilde{\omega}_k$ and simply assume that $\rho_k,\omega_k$ now satisfy \eqref{syst_part2},\eqref{syst_part3} with the modified initial conditions
$$
\displaystyle
\omega_k(0)=\Delta s\varepsilon^{d-1} {\bf a}_1(0,\xi_k),\qquad \rho_k(0) = \Delta s \varepsilon^{d-1} {\bf a}_1(0,\xi_k)g(\xi_k).
$$

Next we compute an approximation of $x_k(s_i), u(x_k(s_i)),\omega_k(s_i)$ and $\rho_k(s_i)$ through the time discretization of the differential systems \eqref{syst_part}, \eqref{syst_part_rho} with an explicit Euler scheme:

\begin{subequations}
	\begin{empheq}[left=\empheqlbrace]{align}
        & x_k^{j+1} = x_k^j + \Delta s \; {\bf a} \left(x_k^j \right),\\
        & \omega_k^{j+1} = \omega_k^j \left( 1+ \Delta s \, (\dive {\bf a}) \left(x_k^j \right) \right), \label{eq_discret_w}\\
	    & \displaystyle \rho_k^{j+1} = \rho_k^{j} - \Delta s \, a_0 \left(x_k^j \right) \rho_k^j + \Delta s \, \omega_k^j S \left(x_k^j \right), \label{eq_discret_rho}\\
		& \displaystyle x_k^0 = (0,\xi_k), \quad u_k^0 = g(\xi_k), \quad \omega_k^0 = \Delta s \varepsilon^{d-1}{ \bf a}_1(0,\xi_k), \quad \rho_k^0 = u_k^0 \omega_k^0.
	\end{empheq}
	\label{eq_syst_part_discret}
\end{subequations}
In order to simplify the discussion, we suppose for simplicity that the function $\dive {\bf a}$ can be computed exactly. Now, we can define the fully discrete approximation of the function $u$:
\begin{equation}
    \tilde u^{\varepsilon,\Delta s}_h(x) :=\sum_{j=n_\ell}^{+\infty}\sum_{k\in\mathbb{Z}^{d-1}}
    \omega_k^j u_k^j \zeta^h \left(x-x_k^j \right) 
    = \sum_{j=n_\ell}^{+\infty}\sum_{k\in\mathbb{Z}^{d-1}}
    \rho_k^j \zeta^h \left(x-x_k^j \right).
    \label{eq_approx_u_h_epsilon_ds}
\end{equation}
\noindent
Note that the sum is finite and is only composed of the indices $(j,k) \in [n_\ell,+\infty[ \times \mathbb{Z}^{d-1}$ such that $x_k^j\in \overline{B}_d(x,h)$.\\

We conclude this section by computing an error estimate associated to the Euler scheme for $x_k,\omega_k,\rho_k$. In the sequel, we will denote respectively $\lfloor\,.\,\rfloor$ the floor function and $\lceil\,.\,\rceil$ the ceiling function.
\begin{lemma}
    Let  $-\ell<t<0<T$ and set $n_t:=\left\lfloor \frac{t}{\Delta s} \right\rfloor$, $N_T:=\left\lfloor \frac{T}{\Delta s} \right\rfloor $.
    Assume ${\bf a} \in C^1(\widetilde{U}_\ell) \cap W^{1,\infty}(\widetilde{U}_\ell)$. Then there exists a constant $C_\Phi(t,T) >0$ such that:
    \begin{equation}
    \displaystyle
        \forall k \in \ZZ^{d-1}, \quad \max_{i\in[n_t, n_T]}|x_k(s_i)-x_k^i| \leq C_\Phi(t,T)\Delta s.
        \label{eq_bound_x(s)-x^j}
    \end{equation}
    If, in addition, ${\bf a}$, $a_0$ and $S$ satisfy Assumption~\ref{assump_reg}, then there exist constants $C_\omega(t,T)>0$, $C_\rho(t,T)>0$ independent of $\Delta s$, $\varepsilon$ and $h$ such that:
    \begin{align}
        \forall i \in [n_t, n_T], \quad \forall k \in \ZZ^{d-1}, \quad 
        \left\{ \begin{array}{cc}
            \displaystyle |\omega_k(s_i)-\omega_k^i| \leq C_\omega(t,T) \Delta s^2 \varepsilon^{d-1},\\
            \displaystyle |\rho_k(s_i)-\rho_k^i| \leq C_\rho(t,T) \Delta s^2 \varepsilon^{d-1}.
            \label{est_cv_w_rho}
        \end{array} \right.
    \end{align}
    \label{lemma_euler_cv}
    Moreover, there exist constants $C_\omega^0(t,T)$ and $C_\rho^0(t,T)$ such that for all $k\in\ZZ^{d-1}$
    \begin{align}
        & \max_{i\in[n_t, n_T]}|\omega_k^i| \leq C_\omega^0(t,T) \Delta s \varepsilon^{d-1},\\
        &\max_{i\in [n_t, n_T]}|\rho_k^i| \leq C_\rho^0(t,T) \Delta s \varepsilon^{d-1}.
    \end{align}
    
\end{lemma}

\begin{proof} Suppose that ${\bf a} \in C^1(\widetilde{U}_\ell) \cap W^{1,\infty}(\widetilde{U}_\ell)$ and let $k\in\ZZ^{d-1}$. From equation~\eqref{vector_field}, one deduces that
$$
\displaystyle
x_k''(s)= \left({\bf a} \cdot \nabla \right) {\bf a}(x_k(s)),\quad |x_k''(s)|\leq \|{\bf a}\|_{L^{\infty}(\widetilde{U}_\ell)}\|\nabla {\bf a}\|_{L^\infty(\widetilde{U}_\ell)},
$$
where ${\bf a}.\nabla = \sum\limits_{i=1}^d { \bf a}_i \partial_{x_i}$.
By computing a Taylor's expansion of $x_k$ on the interval $[s_j, s_{j+1}]$, one gets:
$$
\displaystyle
x_k(s_{j+1}) = x_k(s_j) + \Delta s \, {\bf a}(x_k(s_j)) + \int_{s_j}^{s_{j+1}} x_k''(t)(s_{j+1}-t) dt.
$$
One then obtains the estimate:
\begin{equation*}
\displaystyle
|x_k(s_{j+1})-x_k^{j+1}|\leq 
\left(1+\Delta s\,C(d) \|{\bf a}\|_{W^{1,\infty}(\widetilde{U}_\ell)}\right) |x_k(s_j)-x_k^j| + C(d) \|{\bf a}\|^2_{W^{1,\infty}(\widetilde{U}_\ell)} \frac{\Delta s^2}{2}.
\end{equation*}
Assume that $0 \leq j \leq n_T$ and $x_k^0 = x_k(0)$. By applying a discrete Gronwall argument, one finds that
$$
\displaystyle
\max_{0\leq j\leq n_T} |x_k(s_j)-x_k^j| \leq 
C(d) \|{\bf a}\|^2_{W^{1,\infty}(\widetilde{U}_\ell)} \frac{e^{C(d)\|{\bf a}\|_{W^{1,\infty}(\widetilde{U}_\ell)}T}-1}{2} \Delta s.
$$
By using a similar argument for $n_t\leq j\leq 0$, we get the error estimate \eqref{eq_bound_x(s)-x^j} if we define $C_\Phi(t,T)$ as:
\begin{align}
    \displaystyle
    C_\Phi(t,T) = C(d) \|{\bf a}\|^2_{W^{1,\infty}(\widetilde{U}_\ell)} \frac{e^{C(d)\|{\bf a}\|_{W^{1,\infty}(\widetilde{U}_\ell)} (T-t)}-1}{2}.
\label{eq_lemma_bound_xk_cst}
\end{align}\\

Now we suppose that ${\bf a}, a_0,S$ satisfy the assumption \ref{assump_reg} and we compute an error estimate on 
$e_k^j := \omega_k(s_{j}) - \omega_k^j$, for $n_t \leq j \leq n_T$. From equation~\eqref{syst_part2} and by applying Taylor's theorem, we get:
\begin{align}
    \omega_k(s_{j}) &= \omega_k(s_{j-1}) (1+ \Delta s \, (\dive{\bf a}) (x_k(s_{j-1}))+\Delta s^2\int_0^1 (1-t)\omega_k''(s_{j-1}+t\Delta s) dt.
    \label{eq_taylor_w}
\end{align}
By applying lemma \ref{lemma_bound_wk_rhok}, one obtains
\begin{align*}
    \left|\int_0^1 (1-t)\omega_k''(s_{j-1}+t\Delta s) dt\right| := \left|E^{j-1}_k\right| \leq \underset{u \in [t,T]}{sup} |\omega_k''(u)| \leq C_2(t,T) \Delta s \varepsilon^{d-1}.
\end{align*}
Then, by combining equation~\eqref{eq_discret_w} and \eqref{eq_taylor_w}, one obtains
\begin{align*}
    e_k^{j} &= e_k^{j-1} + \Delta s \left( (\dive {\bf a})(x_k(s_{j-1})) \omega_k(s_{j-1}) - (\dive {\bf a})(x_k^{j-1}) \omega_k^{j-1} \right) + \Delta s^2 E_k^{j-1}\\
    & = e_k^{j-1} + \Delta s \left[ \left( (\dive {\bf a})(x_k(s_{j-1})) - (\dive {\bf a})(x_k^{j-1}) \right) \omega_k(s_{j-1}) - (\dive {\bf a})(x_k^{j-1}) \left( \omega_k(s_{j-1}) - \omega_k^{j-1} \right) \right]\\ 
    & \qquad \quad \, + \Delta s^2 E_k^{j-1}\\
    &= e_k^{j-1} (1 - \Delta s (\dive {\bf a})(x_k^{j-1})) + \Delta s \left( (\dive {\bf a})(x_k(s_{j-1})) - (\dive {\bf a})(x_k^{j-1}) \right) \omega_k(s_{j-1}) + \Delta s^2 E_k^{j-1}.
\end{align*}
As a consequence, by applying Lemma \ref{lemma_bound_wk_rhok}, one deduces that:
\begin{align*}
    |e_k^{j}| & \leq 
    |e_k^{j-1}| \left(1+ \Delta s \|{\bf  a}\|_{W^{1,\infty}(\widetilde{U}_\ell)} \right) + \left(C_\Phi(t,T) C_1(t,T) \| {\bf a}\|_{W^{2,\infty}(\widetilde{U}_\ell)} + C_2(t,T)\right) \Delta s^3 \varepsilon^{d-1}.
\end{align*}
Consequently, by using a discrete Gronwall argument and $e_k^0$, one obtains:
\begin{align*}
    |e_k^j| & \leq \left(C_\Phi(t,T) C_1(t,T) \| {\bf a}\|_{W^{2,\infty}(\widetilde{U}_\ell)}  + C_2(t,T)\right) \Delta s^2 \varepsilon^{d-1} \frac{e^{\|{\bf a}\|_{W^{1,\infty}(\widetilde{U}_\ell)} T}-1}{\|{\bf a}\|_{W^{1,\infty}(\widetilde{U}_\ell)}}.
\end{align*}
One proceeds similarly for $n_t \leq j \leq 0$ and gets the first estimate of \eqref{est_cv_w_rho} by setting
$$
\displaystyle
C_\omega(t,T) = \left(C_\Phi(t,T) C_1(t,T) \| {\bf a}\|_{W^{2,\infty}(\widetilde{U}_\ell)}  + C_2(t,T)\right) \frac{e^{\|{\bf a}\|_{W^{1,\infty}(\widetilde{U}_\ell)} (T-t)}-1}{\|{\bf a}\|_{W^{1,\infty}(\widetilde{U}_\ell)}}.
$$\\

Now we consider the error estimate $l_k^j := \rho_k(s_{j}) - \rho_k^j$, for $n_t \leq j \leq n_T$. By using equation~\eqref{syst_part3} and computing a Taylor expansion of $\rho_k$ on the interval $[s_j, s_{j+1}]$, one finds:
\begin{align}
    \rho_k(s_{j}) &= \rho_k(s_{j-1}) - \Delta s \, a_0(x_k(s_{j-1})) \rho_k(s_{j-1}) + \Delta s \, \omega(x_k(s_{j-1})) S(x_k(s_{j-1})) \label{eq_taylor_rho}\\
    & \quad + \Delta s^2 \int_0^1 (1-t)\rho_k''(s_{j-1}+t\Delta s) dt. \nonumber
\end{align}
From Lemma \ref{lemma_bound_wk_rhok}, one deduces that
\begin{align*}
    \left|\int_0^1 \rho_k''(s_{j-1}+t\Delta s) dt \right|:= L^{j-1}_k \leq \underset{s \in [t,T]}{sup} |\rho_k''(s)| \leq C_3^2(t,T) \Delta s \varepsilon^{d-1}.
\end{align*}
Then, by combining equation~\eqref{eq_discret_rho} and \eqref{eq_taylor_rho}, one has:
\begin{align*}
    |l_k^j| &\leq \Big| l_k^{j-1} - \Delta s \left(a_0(x_k(s_{j-1})) \rho_k(s_{j-1}) - a_0(x_k^{j-1}) \rho_k^{j-1} \right)\Big|\\
    &\quad+\Big|\Delta s \left(S(x_k(s_{j-1})) \omega_k(s_{j-1}) - S(x_k^{j-1}) \omega_k^{j-1} \right)\Big| + \Delta s^2 L_k^{j-1} \\
    & \leq |l_k^{j-1}| \left(1+\Delta s \|a_0\|_{L^\infty(\widetilde{U}_\ell)} \right) + \Delta s^2 C_\Phi(t,T) \|\nabla a_0\|_{L^\infty(\widetilde{U}_\ell)} |\rho_k(s_{j-1})| \\
    & \quad + \Delta s^2 C_\Phi(t,T) |\omega_k(s_{j-1})| \|\nabla S\|_{L^\infty(\widetilde{U}_\ell)} + \Delta s |\omega_k(s_{j-1}) - \omega_k^{j-1} | \|S\|_{L^\infty(\widetilde{U}_\ell)} + C_3^2(t,T) \Delta s^3 \varepsilon^{d-1}\\
    & \leq \left(1+\Delta s \|a_0\|_{L^\infty(\widetilde{U}_\ell)} \right) |l_k^{j-1}| 
    + C_4(t,T)\Delta s^3 \varepsilon^{d-1},
\end{align*}
with 
$$
\displaystyle
C_4(t,T)=\|a_0\|_{W^{1,\infty}(\widetilde{U}_\ell)} C_\Phi(t,T) C_3^0(t,T) + \|S\|_{W^{1,\infty}(\widetilde{U}_\ell)}(C_\Phi(t,T) C_1(t,T) + C_{\omega}(t,T)) + C_3^2(t,T).
$$
Consequently, by the same computations as above, and by $l_k^0=0$, one obtains
\begin{equation}
\displaystyle 
|l_k^j|  \leq C_4(t,T)\frac{e^{\|a_0\|_{L^\infty(\widetilde{U}_\ell)}(T-t)} -1}{\|a_0\|_{L^\infty(\widetilde{U}_\ell)}}\Delta s^2\varepsilon^{d-1}:=C_{\rho}(t,T) \Delta s^2 \varepsilon^{d-1}.
\label{eq_lemma_bound_rhok_cst}
\end{equation}\\


Finally, we derive estimates for $\omega_k^j$ and $\rho_k^j$ with $k\in\ZZ^{d-1}$ and $j\in[n_t,\,n_T]$: it is a straightforward computation to show that
$$
\displaystyle
|\omega_k^{j+1}| \leq \left(1+\Delta s\|{\bf a}\|_{W^{1,\infty}(\widetilde{U}_\ell)}\right)|\omega_k^j|.
$$
As a result, one obtains:
\begin{align*}
\displaystyle
|\omega_k^j|&\leq \exp \left((T-t)\|{\bf a}\|_{W^{1,\infty}(\widetilde{U}_\ell)}\right) \omega_k^0 \\
&\leq \exp \left((T-t)\|{\bf a}\|_{W^{1,\infty}(\widetilde{U}_\ell)}\right) \|{\bf a}\|_{L^{\infty}(\widetilde{U}_\ell)} \Delta s \varepsilon^{d-1} := C_\omega^0(t,T) \Delta s \varepsilon^{d-1}.
\end{align*}\\

We finish the proof of this proposition by deriving an estimate on $\rho_k^j$. One has
$$
\displaystyle
\rho_k^{j+1}\leq (1+\Delta s\|a_0\|_{L^{\infty}(\widetilde{U}_\ell)})\rho_k^j+\|S\|_{L^{\infty}(\widetilde{U}_\ell)} C_\omega^0(t,T) \Delta s^2 \varepsilon^{d-1}.
$$
By a discrete Gronwall argument, one obtains $\rho_k^j\leq C_\rho^0(t,T) \Delta s \varepsilon^{d-1}$ with
\begin{align}
    \displaystyle
    C^0_\rho(t,T) = \|S\|_{L^{\infty}(\widetilde{U}_\ell)} C_\omega(t,T) \frac{e^{\|a_0\|_{L^\infty(\widetilde{U}_\ell)} (T-t) - 1}}{\|a_0\|_{L^\infty(\widetilde{U}_\ell)}}.
    \label{eq_lemma_bound_rhokj_cst}
\end{align}
\end{proof}


 %

\subsection{The main theorem} \label{subsection_4}

\noindent In this section, we prove the convergence of our numerical scheme.
\begin{theorem}
    Let $d\in\mathbb{N}$ such that $d\geq 2$. Suppose that the following conditions are satisfied:
    \begin{enumerate}
        \item Smoothness of data: ${\bf a}$ and $a_0,S,g$ satisfy assumption \ref{assump_reg} and we suppose that 
        $$
        \displaystyle       
        {\bf a}\in C^{d+1}(\widetilde{U}_\ell), \quad g \in C^{d}(\widetilde{U}_\ell),\quad a_0, \, S \in C^d(\widetilde{U}_\ell),\quad \zeta \in C^{d}(\RR).
        $$ 
        \item Characteristic Flow: the flow $\Phi $ satisfies Assumption \ref{assump_flow}. Moreover, we suppose that  
        $$
        \displaystyle
        \nabla \left(\Phi^{-1} \right) \in L^\infty (\widetilde{U}_\ell).
        $$
        \item Numerical parameters: $0\leq h\leq \delta$ and $\varepsilon,\Delta s>0$. Moreover, we assume 
        $$\Delta s=\mathcal{O}_{\varepsilon\to 0}(\varepsilon),\qquad \varepsilon=\mathcal{O}_{h\to 0}(h).
        $$
    \end{enumerate}
    Then, for all compact sets $X \subset \RR^d_+$, there exists $t_X < 0 < T_X$ and 
    a compact set $\Xi_X\subset\RR^{d-1}$
    such that $$\displaystyle \bigcup_{x \in X} \overline{B}_d(x,\delta) \subset \Phi((t_X,T_X) \times \Xi_X).$$ 
    Moreover, there exist constants $C_X^{(0)}, \dots, C_X^{(3)} >0$ which do not depend on $h$, $\varepsilon$, $\Delta s$ such that, if $C_X^{(0)}\,\Delta s  \leq h,$
    the approximation $\widetilde u^{\varepsilon,\Delta s}_h$ defined by equation~\eqref{eq_approx_u_h_epsilon_ds} verifies:
    \begin{align}
        \underset{x \in X}{\sup} |u(x) - \tilde u^{\varepsilon,\Delta s}_h(x)| 
        \leq & C_X^{(1)} h^2 + C_X^{(2)} \frac{\varepsilon^{d}}{h^{d}} + C_X^{(3)} \frac{\Delta s}{h}.
        \label{eq_error_scheme}
    \end{align}
    \label{th_cv_scheme}
\end{theorem}

\begin{remark}
The error term is split into three contributions:
\begin{enumerate}[(i)]
\item The interpolation error: $u(x)-u\ast\zeta^h(x)$.
\item The quadrature error $u\ast\zeta^h(x)-\Pi_h^{\varepsilon,\Delta s}(u)(x)$.
\item The ``time'' discretisation error $\Pi_h^{\varepsilon,\Delta s}(u)(x)-\widetilde{u}_h^{\varepsilon,\Delta s}(x)$ (stemming from the discretization of the characteristic equations). 
\end{enumerate}
\label{remark_terms_th}
\end{remark}

The first assumption of the theorem implies that the flow $\Phi\in C^{d}(\Omega_\ell)$ and that the solution $u\in C^{d}(\widetilde{U}_\ell)$. Now, we consider the first statement of the theorem.  Let us first prove the following lemma.

\begin{lemma} 
    Let $x \in \RR^d_+$, and suppose that $U_\delta=(-\delta, \infty)\times \RR^{d-1}\subset\Phi(\Omega_\ell)$. Then there exists $-\ell < t_x <T_x$, and $\xi  \in \RR^{d-1}, \, \gamma_x >0$ such that
    $$\overline{B}_d(x,\delta) \subset \Phi \left([t_x,T_x] \times \overline{B}_{d-1}(\xi_x, \gamma_x) \right),$$
    (see Figure~\ref{fig_lemma_x_ball} for an illustration).
    \label{lemma_x_ball}
\end{lemma}

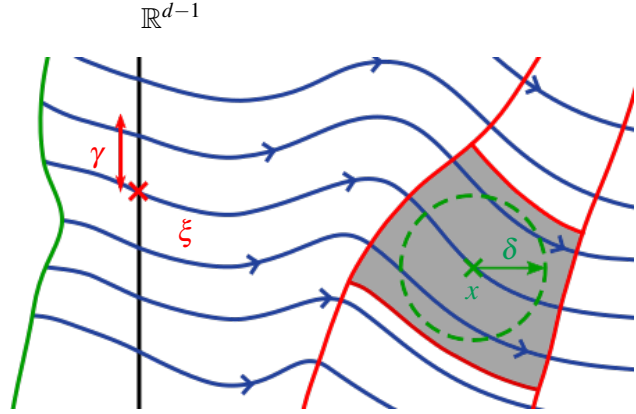
\begin{figure}[ht]
    \center
     \begin{tikzpicture}[scale = 1]
     \node (0,0) {\includegraphics[width=0.5\textwidth]{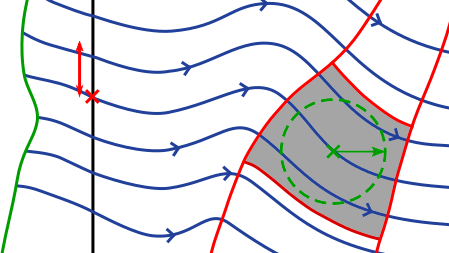}};
     \draw[red] (-1.8,0.) node {$\xi$};
     \draw[red] (-3,1) node {$\gamma$};
     
     \draw[ForestGreen] (2,-0.8) node {$x$};
     \draw[ForestGreen] (2.5,-0.2) node {$\delta$};
     
     \draw[black] (-2,2.9) node {$\RR^{d-1}$};
     \end{tikzpicture}
     
     \caption{Illustration of Lemma \ref{lemma_x_ball}. The set $\overline{B}_d(x,\delta)$ is represented by the green circle. The set $\Phi \left([t_x,T_x] \times \overline{B}_{d-1}(\xi_x,\gamma_x) \right)$ is in grey.}
     \label{fig_lemma_x_ball}
\end{figure}

\begin{proof}
    The set $\overline{B}_d(x,\delta)$ is a compact subset of $\RR^d$ included in $\widetilde{U}_\ell=\Phi(\Omega_\ell)$, and $\Phi : \Omega_\ell \to \widetilde{U}_\ell$ is continuous. Thus $\Phi^{-1}(\overline{B}_d(x,\delta))$ is a compact subset of $\RR^d$ included in $\Omega_\ell$. As a consequence there exists $-\ell < t_x < T_x$, and $\xi_x \in \RR^{d-1}$, $\gamma_x > 0$ such that:
    $$\Phi^{-1} (\overline{B}_d(x,\delta)) \subset [t_x,T_x] \times \overline{B}_{d-1}(\xi,\gamma).$$
    Since $\Phi$ is bijective, this shows the result.
\end{proof}

The proof of the first statement of the theorem is then straightforward: let $K\subset\RR^{d}_+$ be a compact set. For all $x\in K$, there exist $-\ell<t_x<T_x$, $\xi_x\in\RR^{d-1},\gamma_x>0$ such that $\overline{B}_d(x,\delta)\subset\Phi((t_x,T_x)\times B_{d-1}(\xi_x,\gamma_x))$. We then choose 
$$
\displaystyle
t_X=\min_{x\in K}(t_x),\quad T_X=\max_{x\in K}(T_x),\quad \Xi_X=\bigcup_{x\in K}B_{d-1}(\xi_x,\gamma_x)
$$
and we get 
$$\displaystyle \bigcup_{x \in X} \overline{B}_d(x,\delta) \subset \Phi((t_X,T_X) \times \Xi_X).$$ 
\noindent Now we prove the convergence of our numerical scheme.

\begin{proof}
    Let $x\in\RR^d_+$. We first split the difference $u(x)-\tilde u^{\varepsilon,\Delta s}_h(x)$  into four terms:
    \begin{align*}
        u(x) - \tilde u^{\varepsilon,\Delta s}_h(x) =& \,
        u(x) - \left(u \ast \zeta^h \right)(x) \\
        & + \left(u \ast \zeta^h \right)(x)
        - \sum\limits_{j=n_\ell}^{+ \infty} \sum\limits_{k \in \ZZ^{d-1}} \rho_k(s_j) \zeta^h(x-x_k(s_j))\\
        & + \sum\limits_{j=n_\ell}^{+ \infty} \sum\limits_{k \in \ZZ^{d-1}} \left( \rho_k(s_j) - \rho_k^j \right) \zeta^h(x-x_k(s_j))\\
        &+ \sum\limits_{j=n_\ell}^{+ \infty} \sum\limits_{k \in \ZZ^{d-1}} \rho_k^j \left( \zeta^h(x-x_k(s_j)) - \zeta^h \left(x-x_k^j \right) \right) \nonumber\\
        & := I_h(x) + Q_{h,\varepsilon,\Delta s}(x)+ F_{h,\Delta s}^\rho + F_{h,\Delta s}^\zeta(x). \nonumber
    \end{align*}
As explained in Remark \ref{remark_terms_th}, $I_h$ is the interpolation error, $Q_{h,\varepsilon,\Delta s}(x)$ is the quadrature error, and $F_{h,\Delta s}^\rho$ and $F_{h,\Delta s}^\zeta$ are related to the ``time'' discretization errors of the characteristic differential equations. \\

We first consider the interpolation error $I_h(x)$. By applying Lemma \ref{diff-u-uh}, one has
    $$
    \displaystyle
    |I_h(x)|=|u(x)-u\ast\zeta^h(x)|\leq C(x,\delta)h^2, \quad C(x,\delta)=C(d)\|u\|_{W^{2,\infty}(B_d(x,\delta))}.
    $$
As a result, we obtain
$$
\displaystyle
\max_{x\in X}|u(x)-u\ast\zeta^h(x)|\leq C^{(1)}_X\,h^2,\quad C^{(1)}_X=\max_{x\in X}C(x,\delta)
$$
We get $C^{(1)}_X<+\infty$ from the smoothness assumption.\\

Next, we consider the quadrature error $Q_{h,\varepsilon,\Delta s}(x)$. By using the expression \eqref{eq_approx_u_int} and performing the change of variables $s=\frac{\Delta s}{\varepsilon}t$, one finds:
\begin{align*}
    \left( u \ast \zeta^h \right)(x) & =
    \int_{-\ell}^{+ \infty} \int_{\RR^{d-1}} u(\Phi(s,\xi)) (\det J_\Phi) (s,\xi) \zeta^h(x-\Phi(s,\xi)) d\xi ds\\
    &=\int_{-\frac{\varepsilon\ell}{\Delta s}}^{+\infty}\int_{\RR^{d-1}}f\circ g(t,\xi)d\xi dt.
\end{align*}
with 
$$
\displaystyle
g(t,\xi)=\left(\frac{\Delta s}{\varepsilon}t,\xi\right),\quad
f(s,\xi)=\frac{\Delta s}{\varepsilon}u \left(\Phi \left(s,\xi \right) \right) (\det J_\Phi) \left(s,\xi \right) \zeta^h \left(x-\Phi \left(s,\xi \right) \right).
$$
The term $Q_{h,\varepsilon,\Delta s}(x)$ is the error caused by the approximation of this integral in time by
$$\sum\limits_{j=n_\ell}^{+ \infty} \sum\limits_{k \in \ZZ^{d-1}} \rho_k(s_j) \zeta^h(x-x_k(s_j)).$$
We apply Theorem \ref{th_quadrature_Raviard} to the function $(t,\xi) \mapsto f\circ g(t,\xi)$ 
with a mesh of size $\varepsilon$.
This function is $C^d(\Omega_\ell)$ (${\bf a}$ is $C^{d+1}(U)$, thus $\Phi \in C^{d+1}(\Omega_\ell)$ and $\det J_\Phi \in C^{d}(\Omega_\ell)$). As a consequence, one finds:
\begin{align*}
    |Q_{h,\varepsilon,\Delta s}(x)| \leq C(d) \varepsilon^d \|f\|_{W^{d,1}(\Omega_\ell)}.
\end{align*}
We introduce $g(t,\xi) = (\frac{\Delta s}{\varepsilon} t, \xi)$, and $K_x:= g^{-1}(\Phi^{-1}(\overline{B}_d(x,\delta)))$. 
For $\alpha \in \NN^d$ we introduce the differential 
$$D^\alpha_{t,\xi} := \left( \frac{\partial}{\partial t} \right)^{\alpha_1} \left( \frac{\partial}{\partial \xi_1} \right)^{\alpha_2} \dots \left( \frac{\partial}{\partial \xi_{d-1}} \right)^{\alpha_d}.$$
As $\text{supp} (f) \subset K_x $, and by the change of variables $s = \frac{\Delta s}{\varepsilon} t$ we get
\begin{align*}
\displaystyle
    &\|f\circ g\|_{W^{d,1}(\Omega_\ell)} = \|f\circ g \|_{W^{d,1}(K_x)}
    = \sum\limits_{p=0}^d\sum_{\substack{\hphantom{|}\alpha \in \mathbb{N}^d\\ |\alpha|=p}} \int_{K_x} \left|D^\alpha_{t,\xi} f\circ g(t,\xi) \right| d\xi dt\\
    & \; = \frac{\Delta s}{\varepsilon} \sum\limits_{p=0}^d\sum_{\substack{\hphantom{|}\alpha \in \mathbb{N}^d\\ |\alpha|=p}} \int_{\Phi^{-1}(\overline{B}_d(x,\delta))} \left( \frac{\Delta s}{\varepsilon} \right)^{\alpha_1} \left| D_{s,\xi}^\alpha \left(u(\Phi(s,\xi)) (\det J_\Phi)(s,\xi) \zeta^h(x-\Phi(s,\xi)) \right) \right| \frac{\varepsilon}{\Delta s} d\xi ds\\
    & \; \leq \| \det J_\Phi\|_{W^{d,\infty} (\Phi^{-1}(\overline{B}_d(x,\delta)))} \sum\limits_{p=0}^d \sum_{\substack{\hphantom{|}\alpha \in \mathbb{N}^d\\ |\alpha|=p}} \int_{\Phi^{-1}(\overline{B}_d(x,\delta))} \left( \frac{\Delta s}{\varepsilon} \right)^{\alpha_1} \left| D^\alpha_{t,\xi} \left(u(\Phi(s,\xi)) \zeta^h(x-\Phi(s,\xi)) \right) \right| d\xi ds.
\end{align*}
Then, by performing the change of variable associated to $\Phi^{-1}$, one obtains:
\begin{align}
    &\mathcal{X} := \sum\limits_{p=0}^d \sum_{\substack{\hphantom{|}\alpha \in \mathbb{N}^d\\ |\alpha|=p}} \int_{\Phi^{-1}(\overline{B}_d(x,\delta))} \left( \frac{\Delta s}{\varepsilon} \right)^{\alpha_1} \left| D^\alpha_{t,\xi} \left(u(\Phi(s,\xi)) \zeta^h(x-\Phi(s,\xi)) \right) \right| d\xi ds \label{eq_proof_th_X}\\
    & \leq \sum\limits_{p=0}^d \! \sum_{\substack{\hphantom{|}\alpha \in \mathbb{N}^d\\ |\alpha|=p}} \hspace{-1mm} \left( \frac{\Delta s}{\varepsilon} \right)^{\! \! \alpha_1} \hspace{-2mm} \sum_{\substack{\beta \in \mathbb{N}^d;\\ |\beta| \leq |\alpha|}} \!
    \int_{\overline{B}_d(x,\delta)} \!
    \left| D^\beta_{y} \left(u(y) \zeta^h(x-y) \right) \right|
    P_{\alpha \beta} \! \left( |D \Phi|, \dots, |D^{|\alpha|} \Phi| \right) \! \! \left(\Phi^{-1}(y) \right) \left| \det J_\phi^{-1}(y) \right| dy \nonumber
\end{align}
where $P_{\alpha \beta}$ are polynomials with positive coefficients, and $D^k\Phi(y)$ is a vector containing all the partial derivatives $D^\alpha \Phi(y)$ with $|\alpha|=k$.
As $\Delta s=\mathcal{O}_{\varepsilon\to 0}(\varepsilon)$, there exists $A>0$ such that $\Delta s\leq A\varepsilon$ for $\varepsilon$ sufficiently small. Then, by rearranging the terms, we get:
\begin{align*}
    \mathcal{X} & \leq 
    \|\det J_\Phi^{-1}\|_{L^{\infty}(\overline{B}_d(x,\delta))} \sum\limits_{p=0}^d
    \sum_{\substack{\beta \in \mathbb{N}^d\\ |\beta| = p}}
    \int_{\overline{B}_d(x,\delta)} \left| D^\beta_{y} \left(u(y) \zeta^h(x-y) \right) \right|\\
    & \qquad \qquad \times \hspace{-2mm} \sum_{\substack{\hphantom{|}\alpha \in \mathbb{N}^d\\ |\beta| \leq |\alpha| \leq p}}
    A^{\alpha_1} P_{\alpha \beta} \left( |D \Phi|, \dots, |D^{|\alpha|} \Phi| \right) \left(\Phi^{-1}(y) \right) dy\\
    & \leq \|\det J_\Phi^{-1}\|_{L^{\infty}(\overline{B}_d(x,\delta))} C \left( \|\Phi\|_{W^{d,\infty}(\widetilde{U}_\ell)} \right) \|u\|_{W^{d,\infty}(\overline{B}_d(x,\delta))} \|\zeta^h\|_{W^{d,1}(\overline{B}_d(x,\delta))},
\end{align*}
where $C$ is a bounded function on any compact set of $[0,+\infty)$. By definition of $\zeta^h$, for any $p \in \NN$ we have:
$$\|(\zeta^h)^{(p)} \|_{L^1(\RR^d)} = \frac{ \text{meas} (\mathbb{S}^{d-1}) \|\zeta^{(p)}\|_{L^{1}([0,+\infty), \, t^{d-1}dt)} }{ h^p } :=  \frac{C_p}{h^p}.$$
We assumed that  $h\leq \delta$ and thus we get $1/h^p \leq \delta^{d-p} / h^d$. This leads to
\begin{equation}
    |Q_{h,\varepsilon,\Delta s}(x)| \leq C_x \frac{\varepsilon^{d}}{h^{d}},
    \label{eq_bound_Dx}
\end{equation}
where $C_x$ is a constant independent of $\varepsilon, \, \Delta s, \, h$. So, we get the second error term with $C^{(2)}_X = \underset{x \in X}{\max} \; C_x$, which is finite with the smoothness assumption of the theorem~\ref{th_cv_scheme}. \\

Finally, we focus on the error terms $F_{h,\Delta s}^\rho(x)$ and $F_{h,\Delta s}^\zeta(x)$. By Lemma \ref{lemma_x_ball}, there exists $n_x < N_x \in \NN$ and $\xi_x \in \RR^d, \gamma_x > 0$ such that 
$$\overline{B}_d(x,\delta) \subset \Phi \left( [n_x\,\Delta s, N_x\,\Delta s] \times \overline{B}_{d-1}(\xi_x,\gamma_x) \right).$$
We note $t_x=n_x\Delta s$ and $T_x=N_x\Delta s$.
Using Lemma \ref{lemma_euler_cv}, we have for $k \in \ZZ^{d-1}$, $n_x \leq j \leq N_x$:
\begin{align*}
    \displaystyle |\rho_k(s_j)-\rho_k^j| \leq C_\rho(t_x,T_x) \Delta s^2 \varepsilon^{d-1}.
\end{align*}
Moreover, by Lemma \ref{lemma_bound_wk_rhok} (applied to $\tilde{\omega}_k = \Delta s \, \omega_k$ that we now denote by $\omega_k$), 
$$\displaystyle \omega_k(s_j) \geq C_0(t_x,T_x) \Delta s \, \varepsilon^{d-1}.$$
Consequently,
\begin{align}
    |F_{h,\Delta s}^\rho(x)| &= 
    \left|\sum\limits_{j=n_\ell}^{+ \infty} \sum\limits_{k \in \ZZ^{d-1}} \frac{\rho_k(s_j) - \rho_k^j }{\omega_k(s_j)} \omega_k(s_j) \zeta^h(x-x_k(s_j)) \right| \nonumber \\
    & \leq \frac{C_\rho(t_x,T_x)}{C_0(t_x,T_x)} \Delta s \sum\limits_{j=n_\ell}^{+ \infty} \sum\limits_{k \in \ZZ^{d-1}} \omega_k(s_j) \zeta^h(x-x_k(s_j)).
    \label{eq_Ds_rho1}
\end{align}
Now, we can write:
\begin{align}
    \sum\limits_{j=n_\ell}^{+ \infty}  \sum\limits_{k \in \ZZ^{d-1}} & \omega_k(s_j) \zeta^h(x-x_k(s_j)) \nonumber \\
    &= \int_{U} \zeta^h(x-y) dy + \left|\int_{U} \zeta^h(x-y) dy -  \sum\limits_{j=n_\ell}^{+ \infty} \sum\limits_{k \in \ZZ^{d-1}} \omega_k(s_j) \zeta^h(x-x_k(s_j)) \right| \nonumber \\
    & \displaystyle \leq 1 + C(d)\|\zeta\|_{W^{d,1}([0,+\infty), \, t^{d-1}dt)} \frac{\varepsilon^{d}}{h^{d}}
    \leq 1 + C(d) B^d \|\zeta\|_{W^{d,1}([0,+\infty), \, t^{d-1}dt)}.
    \label{eq_Ds_rho2}
\end{align}
The first inequality is obtained by applying the same strategy as for the estimation of $Q$ above, namely using Theorem~\ref{th_quadrature_Raviard}. The second inequality follows from the fifth condition in Theorem~\ref{th_cv_scheme}.
By combining estimates \eqref{eq_Ds_rho1} and \eqref{eq_Ds_rho2}, one obtains:
\begin{align*}
    |F_{h,\Delta s}^\rho(x)| \leq C_x \Delta s \leq C_x \delta \frac{\Delta s}{h},
\end{align*}
for some constant $C_x$ independent of $\Delta s,\varepsilon$ and $h$. The second inequality is obtained by using $h \leq \delta $. Next, we define $C^{(3,1)}_X = \underset{x \in X}{\max} \, \delta C_x$ which is finite thanks to the expressions~\eqref{eq_bound_wk_proof} and~\eqref{eq_lemma_bound_rhok_cst} of $C_0(t_x,T_x)$ and $C_\rho(t_x,T_x)$. We deduce that 
\begin{align}
    \underset{x \in X}{\sup} \,  |F_{h,\Delta s}^\rho(x)| \leq C^{(3,1)}_X \frac{\Delta s}{h}.
    \label{eq_proof_th_bound_3.1}
\end{align}\\

Next, we compute the estimate of $F_{h,\Delta s}^\zeta(x)$. For all $n_x \leq j \leq N_x$, using Lemma \ref{lemma_euler_cv}, we have:
$$|x_k(s_j)-x_k^j| \leq C_\Phi(t_x,T_x)\Delta s .$$
Then, one finds:
\begin{align}
    \left| \zeta^h(x-x_k(s_j)) - \zeta^h \left(x-x_k^j \right) \right| &\leq
    \frac{|x_k(s_j)-x_k^j|}{h^{d+1}} \underset{y \in [0,1] }{\sup} | \zeta'(y)| \nonumber\\
    &\leq C_\Phi(t_x,T_x)\|\zeta'\|_{L^{\infty}([0,1])} \frac{\Delta s}{h^{d+1}}
    := C_{\zeta}(x)\frac{\Delta s}{h^{d+1}}.
    \label{eq_proof_th_cst_zeta}
\end{align}
By using lemma~\ref{lemma_euler_cv} to bound $\rho_k^j$ and that $\zeta^h(x-\cdot)$ is compactly supported in $\overline{B}_d(x,h)$, one finds:
\begin{align}
    \left| F_{h,\Delta s}^\zeta(x) \right| 
    &= \left|\sum_{(k,j) \in E_h(x)} \rho_k^j \left( \zeta^h(x-x_k(s_j)) - \zeta^h \left(x-x_k^j \right) \right) \right|
    \leq C_\zeta(x) C_\rho^0(t_x,T_x) \frac{\Delta s^2}{h^{d+1}} \varepsilon^{d-1} |E_h(x)|,
    \label{eq_proof_th_F}
\end{align}
where 
\begin{align*}
    E_h(x) = \left\{(k,j) \in \ZZ^{d-1} \times \{n_\ell, \dots, +\infty\} \, ; \; x_k(s_j) \text{ or } x_k^j \in \overline{B}_d(x,h) \right\},
\end{align*}
and $|E_h(x)|$ denotes the cardinal of $E_h(x)$.
In order to compute the cardinal of $E_h(x)$, first we introduce the set $E_{2h}^1(x)$, consisting of pairs $(k,j)$ such that $x_k(s_j)$ belongs to the larger set $\overline{B}_d(x,2h)$:
\begin{align*}
    E_{2h}^1(x) = \left\{(k,j) \in \ZZ^{d-1} \times \{n_\ell, \dots, +\infty\} \, ; \; x_k(s_j) \in \overline{B}_d(x,2h) \right\}.
\end{align*}
We show that 
\begin{align*}
    \Phi^{-1}(\overline{B}_d(x,2h)) \subset 
    \overline{B}_d \left(\Phi^{-1}(x), \, 2h \|(\nabla \Phi)^{-1}\|_{L^\infty(\widetilde{U}_\ell)} \right).
\end{align*}
Let us note $(\sigma,\xi) := \Phi^{-1}(x)$, and let $(\sigma',\xi') \in \Phi^{-1}(\overline{B}_d(x,2h))$. Then, $\nabla(\Phi^{-1})$ being bounded, one finds:
\begin{align*}
    \left| (\sigma',\xi') - (\sigma,\xi) \right| &= \left| \Phi^{-1}(\Phi(\sigma',\xi')) - \Phi^{-1}(x) \right|\\
    &\leq \left| \Phi(\sigma',\xi') - x \right| \|(\nabla \Phi)^{-1}\|_{L^\infty(\widetilde{U}_\ell)}
    \leq 2h \|(\nabla \Phi)^{-1}\|_{L^\infty(\widetilde{U}_\ell)},
\end{align*}
since $\Phi(\sigma',\xi') \in \overline{B}_d(x,2h)$. Now, we define by $F_{2h}(x)\supset E_{2h}^1(x)$ the set:
$$
\displaystyle
    F_{2h}(x) = \left\{(k,j) \in \ZZ^{d-1} \times \{-n_\ell, \dots, +\infty\} \, ; \; (s_j,\xi_k) \in \overline{B}_d \left((\sigma,\xi), \, 2h \|(\nabla \Phi)^{-1}\|_{L^\infty(\widetilde{U}_\ell)} \right) \right\},
$$
with $\xi_k$ such that $\Phi(s_j,\xi_k) = x_k(s_j)$. We have:
\begin{align*}
    F_{2h}(x) \subset G_{2h}(x) :=
    & \left[ \left\lfloor \frac{\sigma - 2h \|(\nabla \Phi)^{-1}\|_{L^\infty(\widetilde{U}_\ell)}}{\Delta s} \right\rfloor, \left\lceil \frac{\sigma + 2h \|(\nabla \Phi)^{-1}\|_{L^\infty(\widetilde{U}_\ell)}}{\Delta s} \right\rceil \right]\\
    &\times
    \prod_{1 \leq i \leq d-1} \left[ \left\lfloor \frac{\xi_i - 2h \|(\nabla \Phi)^{-1}\|_{L^\infty(\widetilde{U}_\ell)}}{\varepsilon} \right\rfloor, \left\lceil \frac{\xi_i + 2h \|(\nabla \Phi)^{-1}\|_{L^\infty(\widetilde{U}_\ell)}}{\varepsilon} \right\rceil \right],
\end{align*}

\begin{align*}
    |G_{2h}(x)| &\leq \left(\frac{4h \|(\nabla \Phi)^{-1}\|_{L^\infty(\widetilde{U}_\ell)}}{\varepsilon} +2 \right)^{d-1} \left(\frac{4h \|(\nabla \Phi)^{-1}\|_{L^\infty(\widetilde{U}_\ell)}}{\Delta s} +2 \right)\\
    &\leq \left( 4 \|(\nabla \Phi)^{-1}\|_{L^\infty(\widetilde{U}_\ell)}+ 2B \max(A,1) \right)^d \frac{h^d}{\Delta s \varepsilon^{d-1}} = C(x) \frac{h^d}{\Delta s \varepsilon^{d-1}}.
\end{align*}
As a consequence:
\begin{align}
    |E_{2h}^1(x)| \leq |F_{2h}(x)| \leq |G_{2h}(x)| \leq C(x) \frac{h^d}{\Delta s \varepsilon^{d-1}}.
    \label{eq_bound_E_2h_1}
\end{align}
Let $k \in \ZZ^{d-1}$, $n_x \leq j \leq N_x$, one has:
\begin{align*}
    |x_k^j - x| \geq |x_k(s_j) - x| - |x_k^j - x_k(s_j)|
    \geq |x_k(s_j) - x| - C_\Phi(t_x,T_x) \Delta s.
\end{align*}
Now, we suppose that $C_\Phi(t_x,T_x) \Delta s \leq h$. Then, if $x_k(s_j) \notin \overline{B}_d(x,2h)$, we have $x_k^j \notin \overline{B}_d(x,h)$. Indeed:
\begin{align*}
    |x_k^j - x| \geq 2h - C_x(t_x,T_x)\Delta s  \geq h.
\end{align*}
As a result, $E_h(x) \subset E_{2h}^1(x)$. Thus, we apply equation ~\eqref{eq_bound_E_2h_1} and deduce that 
$$\displaystyle |E_h(x)| \leq C(x)\frac{h^d}{\Delta s \varepsilon^{d-1}}.$$
We define $C^{(0)}_X = \underset{x \in X}{\max} \, C_\Phi(t_x,T_x)$ and $C^{(3,2)}_X = \underset{x \in X}{\max} \, C(d) C_\zeta(x) C_\rho^0(t_x,T_x)$. These quantities are finite by using estimates   \eqref{eq_lemma_bound_xk_cst},~\eqref{eq_proof_th_cst_zeta},~\eqref{eq_lemma_bound_rhokj_cst}. Then, if $C^{(0)}_X \Delta s \leq h$ and by using \eqref{eq_proof_th_F}, one obtains
\begin{align}
    \underset{x \in X}{\sup} \, |F_{h,\Delta s}^\zeta(x)| \leq C_X^{(3,2)} \frac{\Delta s}{h}.
    \label{eq_proof_th_bound_3.2}
\end{align}\\

By combining~\eqref{eq_proof_th_bound_3.1} and~\eqref{eq_proof_th_bound_3.2}, and setting $C_X^{(3)} = C^{(3,1)}_X + C^{(3,2)}_X$, we obtain
\begin{align*}
    \underset{x \in X}{\sup} \, |F_{h,\Delta s}^\rho(x)| + \underset{x \in X}{\sup} \, |F_{h,\Delta s}^\zeta(x)| \leq C_X^{(3)}\frac{\Delta s}{h},
\end{align*}
which leads to the last error term in Theorem~\ref{th_cv_scheme}. This completes the proof of the theorem.

\end{proof}

\section{Numerical tests} \label{sec_test_num}

In this section we present various numerical simulations to test our scheme. Equation \eqref{eq_system} is solved on a bounded domain of $\RR^2_+$, $\Omega := [0,L_{x}] \times [0,L_{y}]$.
We first provide some details on the implementation of the scheme: choice of the kernel function $\eta^h$, boundary conditions and computational costs. Then we present numerical tests on the convergence of the scheme in the linear case. Finally, we use our scheme to carry out numerical simulation of a non linear landscape evolution model.

\subsection{Practical implementation of the scheme}

We provide some details on the implementation of our numerical scheme in two dimensions.

\paragraph{The kernel function $\zeta^h$} We choose the cubic spline kernel, commonly used in the SPH method, as in \cite{monaghan1992smoothed}. For $x \in \RR^2$, $h >0$, this function is given by:
\begin{equation}
    \zeta^h(x) = \frac{10}{7 \pi h^2}
    \left\{ \begin{array}{ll}
        \displaystyle 1 - \frac{3}{2} q^2 + \frac{3}{4} q^3 & \text{if $0 \leq q \leq 1$}, \\
        \displaystyle \frac{1}{4} \left( 2-q^3 \right) & \text{if $1 \leq q \leq 2$}, \\
        \displaystyle 0 & \text{otherwise},
    \end{array} \right.
\end{equation}
where $q = |x|/(4h)$.
This kernel function has compact support, is positive, and belongs to $C^2(\RR^2)$. 

\paragraph{Boundary conditions} In order to simplify the discussion on the treatment of boundary conditions, we assume that the solution $u$ of \eqref{eq_system} is periodic in the transverse direction, that is 
$$ \forall (x,y) \in \Omega, \, \; u(x,y + L_y) = u(x,y).$$
Therefore, $\Omega$ has two boundaries:
$\displaystyle
    \Gamma_1:= \{0\} \times [0,L_y]$ and $\Gamma_2 := \{L_x\} \times [0,L_y].
$
The equation \eqref{eq_system} is a first-order partial differential equation: we only need to fix one boundary condition. To initialize the scheme, the particles $(\xi_k)_{k \in \{1,\dots,n \} }$ are located at the non-characteristic boundary $\Gamma_1$. We first compute the particle positions $x_k^j$ for $-1 \geq j \geq m_k^0$ with the system \eqref{eq_syst_part_discret}, where:
\begin{align*}
    m_k^0 := \max \left\{j \leq -1 \, ; \;  \left(x_k^j \right)_1 < -h \right\} +1,
\end{align*}
that is, the position of particles outside $\Omega$ that are located at a distance smaller than $h$ to $\Gamma_1$. 
Then, we compute the position of the particles inside the domain $\Omega$. These are the particles $x_k^j$ such that $0 \leq j \leq m_k^1$ with
\begin{align*}
    m_k^1 := \min \left\{j \geq 1 \, ; \;  \left(x_k^j \right)_1 > L_x \right\} -1 .
\end{align*}
Finally, we compute the position of the particles 
$x_k^j$ such that $m_k^1 \leq j \leq m_k^2$ in the domain $]L_x,L_x+h] \times [0,L_y]$, located to the right of the right boundary $\Gamma_2$ with $m_k^2$ defined by:
\begin{align*}
    m_k^2 := \min \left\{j \geq 1 \, ; \;  \left(x_k^j \right)_1 > L_x+h \right\} -1.
\end{align*}
We refer to Fig~\ref{fig_illustr_scheme} for an illustration of the particles positions.

\paragraph{Computation of the approximate solution}

Formula \eqref{eq_approx_u_h_epsilon_ds} and system~\eqref{eq_syst_part_discret} are used to  compute the approximation $\widetilde{u}_h^{\varepsilon,\Delta s}$ of $u$ in the domain $\Omega$, and in its extension $[-h,0[ \times [0,L_y]$ and $]L_x,L_x+h] \times [0,L_y]$. We denote by $\tilde \Omega := [-h,L_x+h] \times [0,L_y]$. We compute the approximate solution at points $x$ on a grid fine enough to provide a good insight into the variations of the approximate solution $\tilde u^{\varepsilon,\Delta s}_h$. One has to compute the weights $\omega_k^j$ and $\rho_k^j$. For $1 \leq k \leq n$, $\omega_k^j$ verifies \eqref{eq_discret_w}. The function $\dive {\bf a}$ is not analytically known, so it needs to be approximated numerically. Consequently, we create a rectangular mesh of $\tilde \Omega$, and compute the derivatives of $\bf a$ by a finite difference scheme. Then, we use an linear interpolation to determine the value of $\dive{\bf a}$ at the particle positions. The computation of $\rho_k^j$ follows easily from \eqref{eq_discret_rho}.

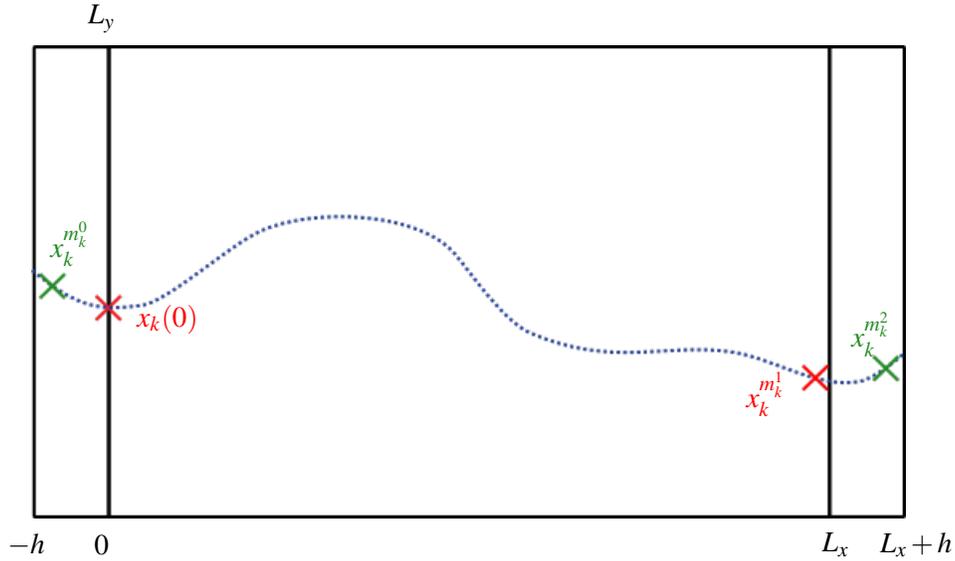
\begin{figure}[ht]
    \center
     \begin{tikzpicture}[scale = 1]
     \node (0,0) {\includegraphics[width=0.7\textwidth]{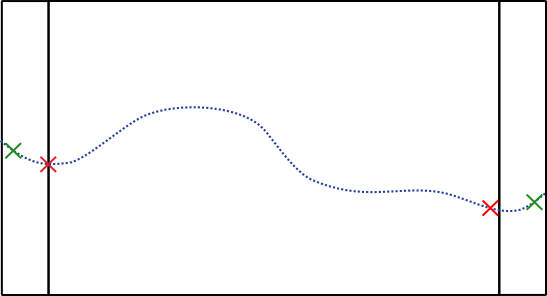}};
     \draw[red] (-4,-0.5) node {$x_k(0)$};
     \draw[my_green] (-5.3,0.5) node {$x_k^{m_k^0}$};
     \draw[red] (3.95,-1.5) node {$x_k^{m_k^1}$};
     \draw[my_green] (5.35,-0.7) node {$x_k^{m_k^2}$};
     
     \draw[black] (-4.88,3.5) node {$L_y$};
     \draw[black] (-4.88,-3.5) node {$0$};
     \draw[black] (-5.88,-3.5) node {$-h$};
     \draw[black] (4.88,-3.5) node {$L_x$};
     \draw[black] (5.95,-3.5) node {$L_x+h$};
     \end{tikzpicture}
     
     \caption{Illustration of particle positions in the scheme. Particles are generated along the time-discretization of the blue dotted curve (a characteristic curve). The red points are the extremal particles inside the domain, and the green one are the extremal particles outside it.}
     \label{fig_illustr_scheme}
\end{figure}


\paragraph{Reduction of the computational cost}

For any $x \in \Omega$, the approximate solution $\tilde u^{\varepsilon,\Delta s}_h$ is given by:
\begin{align}
    \tilde u^{\varepsilon,\Delta s}_h(x) = \sum\limits_{k = 1}^n \sum\limits_{j = m_k^0}^{m_k^2} \rho_k^j \zeta^h \left(x-x_k^j \right).
    \label{eq_numTest_u}
\end{align}
In this sum, only a few terms are not zero, as the support of $\zeta^h(x-\cdot)$ is restricted to $\overline{B}_d(x,h)$. Therefore, computing
every quantity $\zeta^h (x-x_k^j)$ for all values of $k$ and $j$ is unnecessary, and can be improved. 
For that purpose, we used the cell linked-lists method~\cite{dominguez2011neighbour}. This method consists of two steps:
\begin{enumerate}
    \item First, a mesh grid with a spatial step $h$ is created in the domain, together with an array representing this grid. Each particle is positioned in its corresponding cell inside this mesh grid. 
    \item If $x \in \Omega$, the particles that interact with this point (i.e., those corresponding to a non-zero term in the sum \eqref{eq_numTest_u}) are located in the cell containing $x$ and its eight neighbours. Therefore, the sum is computed only on the particles that belongs to these nine cells, using the previous array.
\end{enumerate}

\begin{figure}[ht]
    \center
    \includegraphics[width=0.55\textwidth]{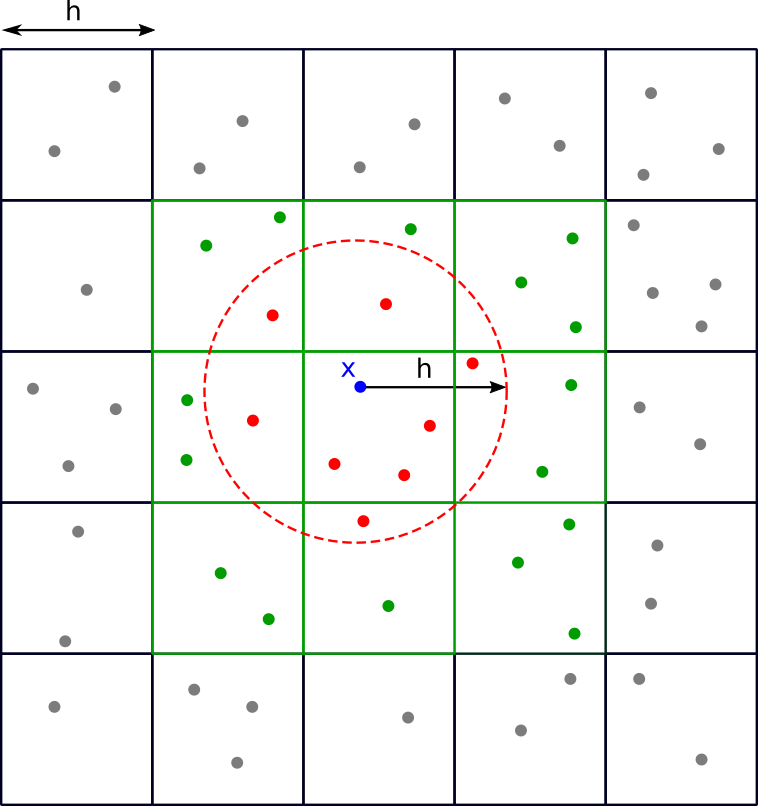}
    \caption{Grid of the linked-list method, of size $h$. The particles in the nine green neighboring cells of the blue point $x$ are in color. The red particles are located in the support of $y \mapsto \zeta^h(x-y)$ (represented by the red circle), the green particles are inside the green cells but outside this support.}
    \label{fig_linked_list}
\end{figure}

\noindent The array of particle positions in the grid needs to be computed only once. Figure~\ref{fig_linked_list} represents the grid and the particles  in the domain. The red particles interact with the point $x$, the green domain is composed of the nine neighboring cells. Green particles add unnecessary computations, but are not too many compared to the total number of particles in the domain.  A direct implementation of \eqref{eq_numTest_u} would require $N$ operations, where $N$ is the number of particles on the domain. Using the cell linked-list method, the computation of the sum requires approximately $9Nh^2\ll N$ operations (equals to the number of particles in the nine neighboring cells to the point $x$). The approximate solution $\widetilde{u}_h^{\varepsilon,\Delta}$ is evaluated on a grid with $N_x=L_x/\Delta x$ and $N_y=L_y=L_y/\Delta y$ in the streamwise and cross-stream directions. The total computational cost is then approximately $9Nh^2N_xN_y$ with the cell linked list.




\subsection{Numerical study of convergence}

In this section we present numerical experiments using the particle scheme in the domain $\Omega$ and study the convergence. With a view to landscape evolution models simulation, we consider the following linear system:
\begin{equation}
    \left\{ \begin{array}{l}
        \displaystyle \dive \left( {\bf a} u \right)(x,y) = r(x,y),\\
        u(0,y) = u_0(y),
    \end{array} \forall (x,y) \in \tilde \Omega\right.
    \label{eq_syst_linear_h}
\end{equation}
We define ${\bf a} := ({\bf a}_x,{\bf a}_y)$ with
\begin{equation}
    \left\{ \begin{array}{l}
        {\bf a}_x(x,y) = \tan \theta - \dx z(x,y), \\
        {\bf a}_y(x,y) = -\dy z(x,y).
    \end{array} \right.
    \label{eq_a_linear_inside}
\end{equation}
For $(x,y) \in \tilde \Omega \backslash \Omega$, the vector field ${\bf a}$ is given by:
\begin{equation}
    \left\{ \begin{array}{l}
        {\bf a}_x(x,y) = \tan \theta, \\
        {\bf a}_y(x,y) = 0.
    \end{array} \right.
    \label{eq_a_linear_outside}
\end{equation}
The function $(x,y)\mapsto \tan \theta (L_x-x) + z(x,y)$ will represent the bottom surface height in the landscape evolution model. When $z=0$, this surface is a tilted plane, $\theta$ being the angle formed by this plane with the horizontal plane. This system describes the stationary distribution of fluid height $u$ on a given topography when a constant fluid flow is injected at the boundary $x=0$.

\begin{table}[ht]
\center 
\begin{tabular}{| l| l | l |}
    \hline
    Length of the domain & $L_x$ & $1$\\
    Width of the domain & $L_y$ & $0.25$\\
    Plane inclination & $\theta$ & $39^\circ$\\
    Water height at the boundary & $u_0(y)$ & $10^{-3}$\\
    Source term & $r(x,y)$ & Variable\\
    Surface perturbation & $z(x,y)$ & Variable\\
    \hline  
\end{tabular}
\caption{Parameters of numerical simulations}
\label{tab_param}
\end{table}

In what follows, we solve system \eqref{eq_syst_linear_h} with our Particle method
for various values of the parameters. The modeling parameters are given in Table~\ref{tab_param}, taken or estimated from an experiment~\cite{guerin_streamwise_2020} with a block of plaster set on a small tilted plane and eroded by a layer of water with a constant height. The numerical parameters are given in Table~\ref{tab_param_num}. We consider that equation~\eqref{eq_syst_linear_h} is written in a non-dimensional form, so that these parameters have no physical units.

\begin{table}[ht]
\center 
\begin{tabular}{| l| l | l |}
     \hline
     Diameter of the kernel function & $h$ & Variable\\
     Distance between particles at the inbound boundary & $\varepsilon$ & Variable\\
     Discretization step for variable $s$ & $\Delta s$ & Variable\\
     Weight of particles at the boundary & $w_0$ & $\tan \theta \, \Delta s \, \varepsilon$\\
     \hline  
\end{tabular}
\caption{Parameters of the scheme.}
\label{tab_param_num}
\end{table}
 
 \noindent In order to test the convergence of the scheme, three numerical parameters are tuned: $h$, $\varepsilon$ and $\Delta s$. 
 The weights $w_0$ at $x=0$
 are the product of the distance between two particles, $\varepsilon$, and the distance traveled by a particle during one time step: ${\bf a}(0,y) \, \Delta s = \tan \theta \, \Delta s$. This gives $w_0 := \tan \theta \, \Delta s \, \varepsilon$. \\
 
\noindent We consider the following values for $h$, $\varepsilon$ and $\Delta s$:
\begin{align*}
    h_{char} = \frac{1}{40}, \quad \varepsilon_{char} =\frac {h_{char}}{10}, \quad \Delta s_{char} = \frac{1}{200}.
\end{align*}
The values of $h$, $\varepsilon$ and $\Delta s$ can vary in the following range:
\begin{align*}
    & h_{min} := 2^{-4} h_{char} \leq h \leq h_{max} := 2^{4} h_{char}, \\
    & \varepsilon_{min} :=2^{-2} \varepsilon_{char} \leq \varepsilon \leq \varepsilon_{max} := 2^3 \varepsilon_{char},\\
    & \Delta s_{min} := 2^{-5} \Delta s_{char} \leq \Delta s \leq \Delta s_{max} := 2^5 \Delta s_{char}.
\end{align*}
These bounds are chosen to ensure that the support of kernel is smaller than the domain, and that there is an average of more than one particle in the kernel function support. We present some error curves, depending on parameters. The error is the $\ell^\infty$ relative error, defined by:
\begin{align}
    \displaystyle \text{Error } = 
    \frac{ \underset{1 \leq i \leq N_x, 1 \leq j \leq N_y}{\max} \left( \tilde u^{\varepsilon,\Delta s}_h(i \,dx, j\, dy) - u(i \,dx, j\, dy) \right)}{ \underset{1 \leq i \leq N_x, 1 \leq j \leq N_y}{\max}u(i \,dx, j\, dy) },
    \label{eq_l2_relative_error}
\end{align}
where $dx$, $dy >0$ are given by $dx \, N_x = L_x$ and $dy \, N_y = L_y$.

\subsubsection{Numerical convergence of the Particle method}
 
\paragraph{Test on a flat surface.}

First we study the convergence of the scheme in the simplest framework $z=0$. The upstream water height is constant, of value one and the right hand side $r$ is set to $0$. Thus the exact solution is a constant water height, equal to one.
The numerical results performed with $h=h_{char}$, $\varepsilon=\varepsilon_{char}$ and $\Delta s =\Delta s_{char}$ produce a relative error with the exact solution of about $0.14\%$. 
As $u$ and ${\bf a}$ are constants, there is neither convolution error nor error in the computation of the $\omega_k^j$ and $\rho_k^j$. The only error comes from the quadrature error and according to Theorem \ref{th_cv_scheme}, we have:
\begin{align}
    \text{Error } \leq \frac{C \varepsilon^2}{h^2}.
    \label{eq_error_constant_height}
\end{align}
Thus the error grows with $\varepsilon /h$. \newline
\begin{figure}[ht]
     \centering
     \begin{subfigure}[b]{0.45\textwidth}
        \includegraphics[width=\textwidth]{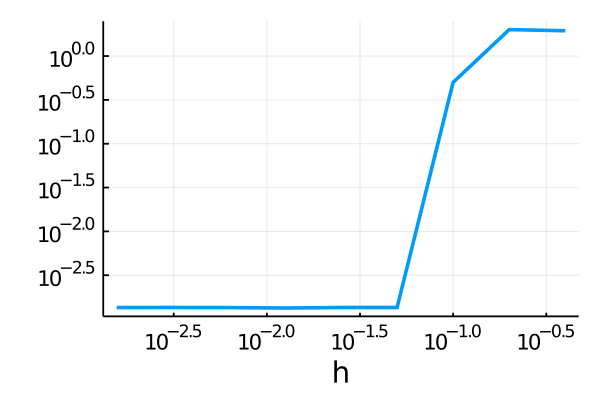}
        \caption{$\varepsilon= h/10$, $\Delta s= 4 \varepsilon$.}
        \label{fig_water_error1_1}
    \end{subfigure}
    \begin{subfigure}[b]{0.45\textwidth}
        \includegraphics[width=\textwidth]{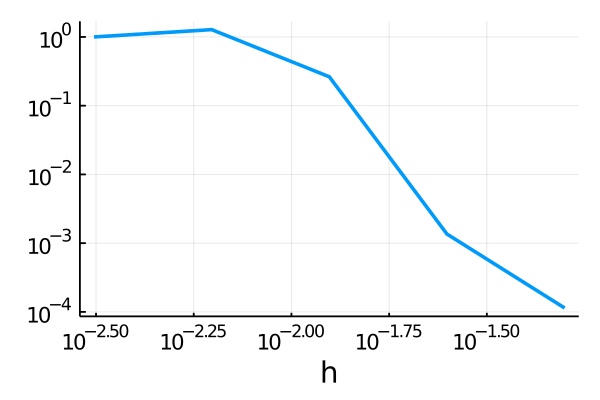}
        \caption{$\varepsilon= h_{char}/10$ and $\Delta s= 4 \varepsilon$ fixed.}
        \label{fig_water_error1_2}
    \end{subfigure}\\
    \centering
    \begin{subfigure}[b]{0.45\textwidth}
        \includegraphics[width=\textwidth]{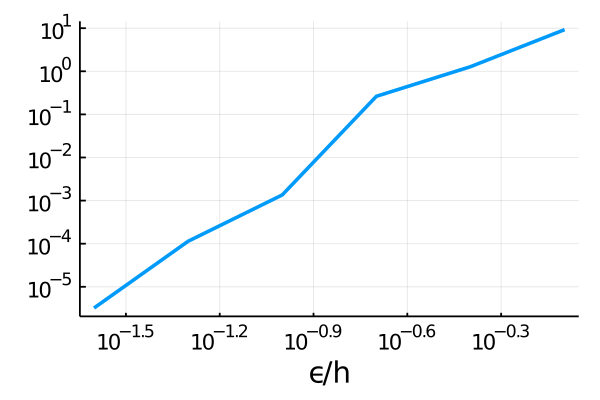}
        \caption{$h=h_{char}$ fixed and $\Delta s= 4 \varepsilon$.}
        \label{fig_water_error1_3}
    \end{subfigure}
    \begin{subfigure}[b]{0.45\textwidth}
        \includegraphics[width=\textwidth]{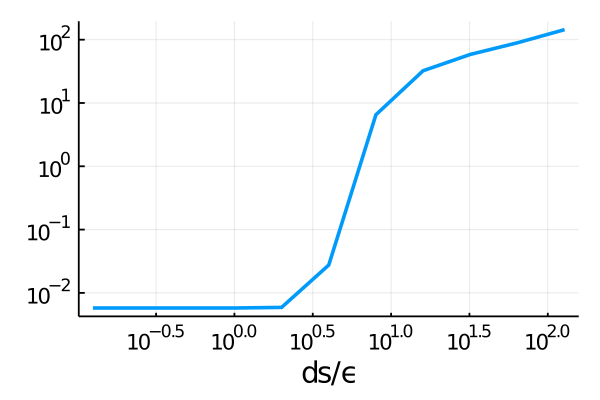}
        \caption{$h=h_{char}$ and $\varepsilon=\varepsilon_{char}$ fixed.}
        \label{fig_water_error1_4}
    \end{subfigure}
    \caption{Flat surface test case: graphs of the $\ell^\infty$ relative errors \eqref{eq_l2_relative_error} as functions of $h$, $\varepsilon$ and $\Delta s$, in log-log scale (using decimal logarithms). $z=0$, $r = 0$.}
    \label{fig_water_error1}
\end{figure}

In figure~\ref{fig_water_error1_1}, we vary $h$  between $h_{\min}$ and $h_{\max}$, and the values of $\varepsilon$ and $\Delta s$ are varied accordingly with the formula: $\varepsilon=h/10$ and $\Delta s=4\varepsilon$. Therefore, the number of particles in $\overline{B}_d(x,h)$ does not change. The computation times for the different values of $h$ remain almost constant. This is due to the linked-list method which keeps constant the number of computations when the number of particles in $\overline{B}_d(x,h)$ is constant. The only computation which takes more operations when $h$ becomes smaller is the computation of the linked-list grid. 

\smallskip

For values of $h$ greater than $1/40$ the error grow with $h$. Indeed, a more precise formula for~\eqref{eq_error_constant_height} would be
\begin{align}
    \text{Error } \leq C \varepsilon^2 \left(1 + \frac{1}{h}+  \frac{1}{h^2} \right),
    \label{eq_error_constant_height2}
\end{align}
(this can be seen by ignoring the simplification made just before~\eqref{eq_bound_Dx} in the proof of Theorem~\ref{th_cv_scheme}). The term $C \left(1+1/h \right)$ is responsible for this growth. However, this effect only appears when $h$ is large enough (see below). In practice the value of $h$ is small and this effect does not appear. Indeed, for values of $h$ smaller than $1/40$ the error is constant, of about $10^{-2.9}$, which is consistent with the error bound~\eqref{eq_error_constant_height}. 

\smallskip
Figure~\ref{fig_water_error1_2} shows the relative error depending on $h$, when $\varepsilon$ and $\Delta s$ are fixed. In accordance with equation~\eqref{eq_error_constant_height} this error grows when $h$ becomes smaller, because the number of particles in the kernel diminishes with $h$. Figure~\ref{fig_water_error1_3} shows the error when $h=h_{char}$ and $\varepsilon$ varies with $\Delta s$. It shows that the error decreases if $\varepsilon$ and $\Delta s$ are decreased in agreement with equation~\eqref{eq_error_constant_height}. In figure~\ref{fig_water_error1_4}, $h=h_{char}$, $\varepsilon=\varepsilon_{char}$ and only $\Delta s$ varies. The error decreases with $\Delta s$, and reaches a minimum of about $10^{-5}$, when $\Delta s \leq 2 \varepsilon$. This is due to a simplification made in the proof of the theorem~\ref{th_cv_scheme} where we used the inequality $\Delta s\leq A\varepsilon$ in the proof of the quadrature error, after equation~\eqref{eq_proof_th_X}. Therefore a more precise formula for~\eqref{eq_error_constant_height} contains terms that grow with $\Delta s / \varepsilon$, which is consistant with figure~\ref{fig_water_error1_4}. This shows that the spacing between quadrature points along $s$ and along $\xi$ should be roughly equal.

\smallskip
In figure~\ref{fig_water_error1_3}, we observe a minimum relative error on the order of $10^{-5.3}$, but which could be made smaller with smaller values of $\varepsilon$. Setting $h=h_{char}$, $\varepsilon=h/40$ and $\Delta s= 2\varepsilon$, the $\ell^\infty$ relative error is $7 \times 10^{-7}$. Of course, the computation time dramatically increases with such a small value of $\varepsilon$.

\bigskip
We have also considered a non-zero constant source term $r=1/20$. The exact solution can be computed explicitly:
\begin{align*}
    \forall (x,y) \in \Omega, \quad u(x,y) = u_0 + \frac{r \, x}{\tan \theta}.
\end{align*}
We found that the $\ell^\infty$ relative error when $h=h_{char}$, $\varepsilon=\varepsilon_{char}$, and $\Delta s = \Delta s_{char}$ is $2.9 \times 10^{-3}$.

\paragraph{Test on a non-flat surface, transverse perturbation.}

We consider the system with a surface varying in the transverse direction $y$:
\begin{align*}
    \forall (x,y) \in \Omega, \quad z_1(x,y) = \frac{u_0}{2} \cos \left(\frac{2 \pi}{L_y} \left(y-\frac{L_y}{2} \right) \right).
\end{align*}

\begin{figure}[ht]
     \centering
     \begin{subfigure}[b]{0.32\textwidth}
        \includegraphics[width=\textwidth]{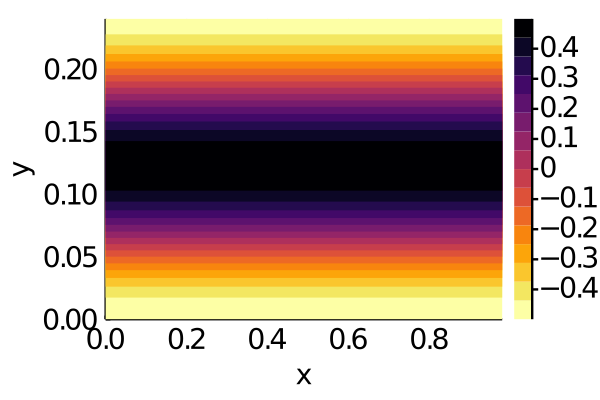}
        \caption{Surface height.}
        \label{fig_water_2_soil}
     \end{subfigure}
     \begin{subfigure}[b]{0.32\textwidth}
        \includegraphics[width=\textwidth]{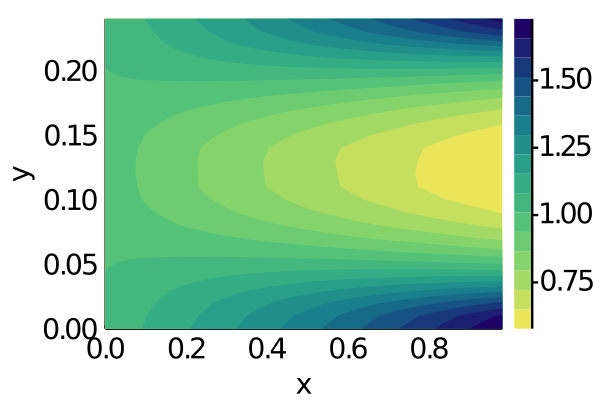}
        \caption{Water height.}
        \label{fig_water_2_water}
     \end{subfigure}
     \begin{subfigure}[b]{0.32\textwidth}
         \includegraphics[width=\textwidth]{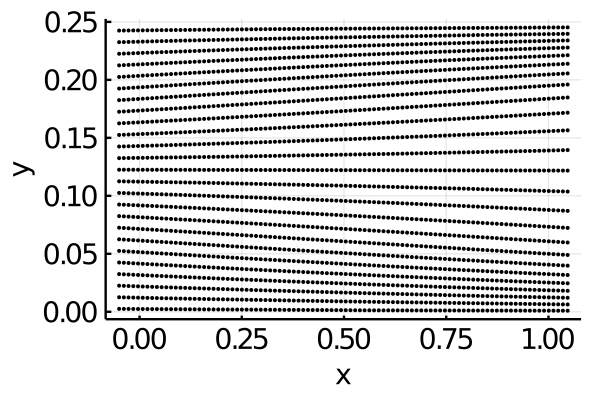}
         \caption{Particle positions.}
         \label{fig_water_2_part}
     \end{subfigure}
     \caption{Non-flat surface,  transverse perturbation: surface height~(a), water height~(b) and particle positions~(c) in the domain $\Omega$, with $z=z_1$, $r = 0$ and $h = 4 h_{char}$, $\varepsilon= \varepsilon_{char}$, $\Delta s= 2 \Delta s_{char}$.}
     \label{fig_water_2}
\end{figure}

\noindent Figure~\ref{fig_water_2_soil} shows a heat map of this function. Figure~\ref{fig_water_2_water} presents the water height computed by the particle scheme, and particle positions are showed in Figure~\ref{fig_water_2_part}. We can see that, when $x$ increases, particles tend towards the sides $y=0$ and $y=0.25$, where the surface height is smaller.

\begin{figure}[ht]
     \centering
     \begin{subfigure}[b]{0.45\textwidth}
        \includegraphics[width=\textwidth]{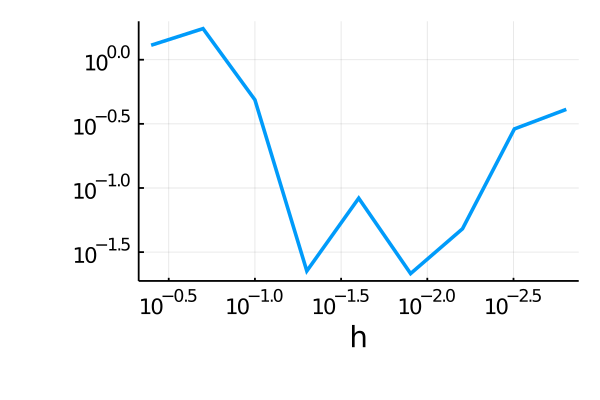}
        \caption{$\varepsilon= h/10$ and $\Delta s= 4 \varepsilon$.}
        \label{fig_water_error2_1}
    \end{subfigure}
    \begin{subfigure}[b]{0.45\textwidth}
        \includegraphics[width=\textwidth]{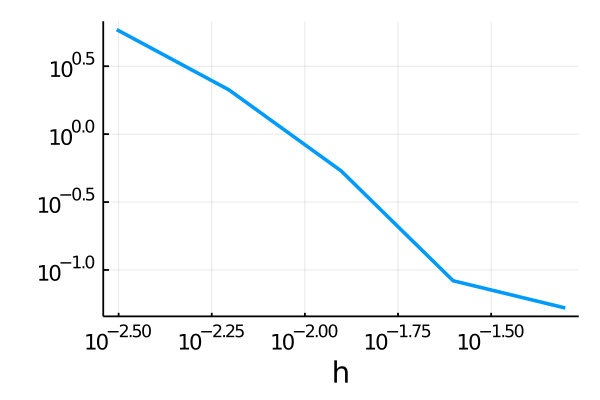}
        \caption{$\varepsilon= h_{char}/10$ and $\Delta s= 4 \varepsilon$ fixed.}
        \label{fig_water_error2_2}
    \end{subfigure}\\
    \centering
    \begin{subfigure}[b]{0.45\textwidth}
        \includegraphics[width=\textwidth]{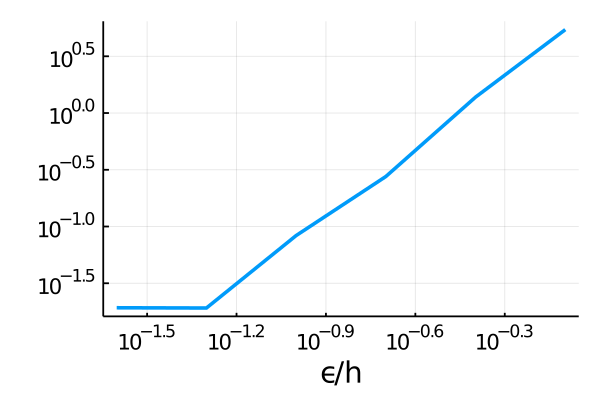}
        \caption{$h=h_{char}$ fixed and $\Delta s= 4 \varepsilon$.}
        \label{fig_water_error2_3}
    \end{subfigure}
    \begin{subfigure}[b]{0.45\textwidth}
        \includegraphics[width=\textwidth]{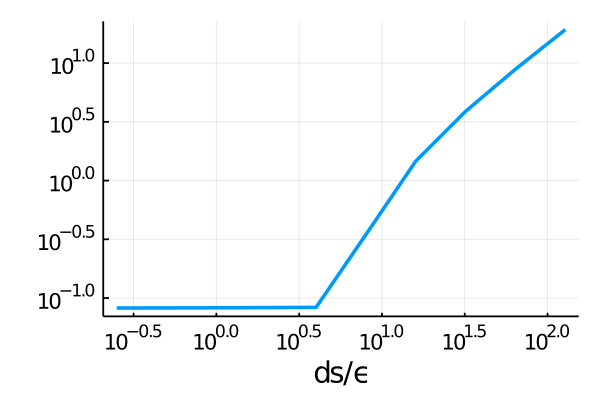}
        \caption{$h=h_{char}$ and $\varepsilon=\varepsilon_{char}$  fixed.}
        \label{fig_water_error2_4}
    \end{subfigure}
    \caption{Non-flat surface,  transverse perturbation: graphs of the $\ell^\infty$ relative errors \eqref{eq_l2_relative_error} as functions of $h$, $\varepsilon$ and $\Delta s$, in log-log scale. $z=z_1$, $r = 0$.}
    \label{fig_water_error2}
\end{figure}

\medskip
Next, we present some error curves.
The exact solution of System~\eqref{eq_syst_linear_h} cannot be computed. Therefore we compute a reference solution with a finite-volume method as in~\cite{binard2024well}, and a discretisation of $400 \times 100$ points; the error is computed using this reference solution.
First, in Figure~\ref{fig_water_error2_1}, $h$ varies with $\varepsilon/h$ and $\Delta s/h$ is kept constant.
For $1/3 \geq h \geq 1/20$, the error decreases with $h$, due to the first term of the error bound~\eqref{eq_error_scheme}. To decrease the error, $\varepsilon$ or $\Delta s$ should be taken small compared to $h$. 
Then, Figure~\ref{fig_water_error2_2} shows the relative error depending on $h$, when $\varepsilon$ and $\Delta s$ are fixed. This error grows when $h$ becomes smaller, because of the terms in powers of $1/h$ in Equation~\eqref{eq_error_scheme}.
Figure~\ref{fig_water_error2_3} shows the error when $h=h_{ref}$ is fixed and $\varepsilon$, $\Delta s$ vary, and in Figure~\ref{fig_water_error2_4} only $\Delta s$ varies. In both Figures, the error decreases when $\varepsilon$ and $\Delta s$ diminish, because of the last two terms of Equation~\eqref{eq_error_scheme}.

\paragraph{Test on a non-flat surface, longitudinal perturbation.}

Now, we consider the system with a surface varying in $x$, in the longitudinal direction:
\begin{align}
    \forall (x,y) \in \Omega, \quad z_2(x,y) = 20 \, u_0 \cos \left(\frac{8 \pi}{L_x} x \right).
    \label{eq_function_z2}
\end{align}

\begin{figure}[ht]
     \centering
     \begin{subfigure}[b]{0.4\textwidth}
        \includegraphics[width=\textwidth]{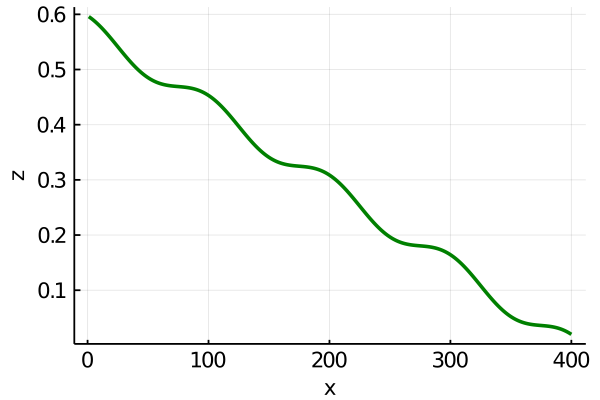}
        \caption{Surface profile as a function of $x$.}
        \label{fig_3_profile}
     \end{subfigure} 
     \begin{subfigure}[b]{0.4\textwidth}
        \includegraphics[width=\textwidth]{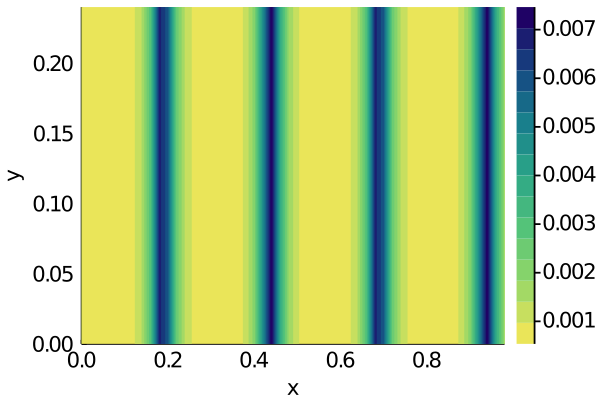}
        \caption{Water height.}
        \label{fig_3_water}
     \end{subfigure}
     \caption{Non-flat surface, longitudinal perturbation: surface profile~(a) and water height~(b). $z=z_2$, $r = 0$, $h = 4 h_{char}$, $\varepsilon= \varepsilon_{char}$, $\Delta s= 2 \Delta s_{char}$}
     \label{fig_3}
\end{figure}

\noindent The profile of this surface is showed in Figure~\ref{fig_3_profile}. We can see that the function $x \mapsto z_2(x,0)$ is strictly decreasing. Thus the vector field ${\bf a}$ does not vanish.
Figure~\ref{fig_3_water} presents the water height computed by the particle scheme.
The quantitative behavior of the error curves is similar to the previous example~\ref{fig_water_2}, but the errors are approximately fifty times bigger. It can be explained by the fact that the minimum of the function $\det J_\Phi$ is quite small, and this term appears in the error estimate \eqref{eq_error_scheme}. As a comparison, the relative error is ten times smaller if the function $z$ is divided by two compared with~\eqref{eq_function_z2}, when $h=h_{char}$, $\varepsilon= \varepsilon_{char}$, and $\Delta s= 2 \Delta s_{char}$.


\subsubsection{Numerical simulations with dry areas}

We have designed a numerical scheme under the assumption that the characteristic curves cover entirely the computational domain. Using the terminology of landscape evolution models, this means that there are no ''dry areas" in the domain. Here, we discuss an extension of the numerical scheme in the presence of ``dry areas''.

\begin{definition}
    A dry area is a connected component of $\widetilde{U}_\ell \backslash \overline{\Phi(\Omega_\ell)}$.
    That is, it is a set of points that do not belong to a characteristic curve issued from the boundary $x_1=0$. We denote by $D_{\widetilde{U}_\ell}$ the union of all the dry areas of $\widetilde{U}_\ell$, that is:
    $$D_{\widetilde{U}_\ell} = \widetilde{U}_\ell \backslash \overline{\Phi(\Omega_\ell)}.$$
\end{definition}

In the case of dry areas, we may construct a solution on the whole domain if we assume that the source term $S$ of \eqref{eq_system} vanishes on $D_{\widetilde{U}_\ell}$. We then consider the solution, given by:
\begin{itemize}
    \item Equation~\eqref{eq_expr_sol_u} in $\text{Int}(\Phi(\Omega_\ell))$,
    \item $u=0$ in $D_{\widetilde{U}_\ell}$.
\end{itemize}
The solution may be discontinuous on $\partial (\Phi(\Omega_\ell)) = \partial (D_{\widetilde{U}_\ell})$ and we do not define it on this set.
See figure~\ref{fig_caract_curves_dry} for an example of such a situation: the dashed area is not reached by characteristic curves. It can be seen as a hill, whose top is the red cross. The left pointed edge of the dashed area is a mountain pass, and ${\bf a}$ vanishes at this point.

\begin{figure}[ht]
    \center
     \begin{tikzpicture}[scale = 1]
     \node (0,0) {\includegraphics[width=0.5\textwidth]{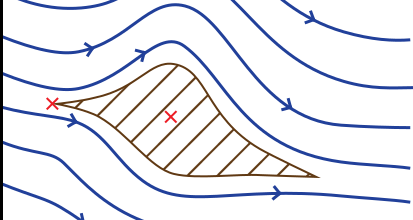}};
     \draw[black] (-4,2.6) node {$\RR^{d-1}$};
     \end{tikzpicture}
     
     \caption{Characteristic curves on the domain $\RR^{d-1} \times \RR^+$. The vector field ${\bf a}$ vanishes at points indicated by red crosses.}
     \label{fig_caract_curves_dry}
\end{figure}

We have considered from a numerical point of view and chosen a surface with more variations than in previous section, see figure~\ref{fig_4_soil}. For $(x,y) \in \Omega$, let
\begin{align*}
    z_3(x,y) = 2 u_0 \cos \left(\frac{4 \pi}{L_y} \left(y-\frac{L_y}{2} \right) \right).
\end{align*}

\begin{figure}[ht]
     \centering
     \begin{subfigure}[b]{0.33\textwidth}
        \includegraphics[width=\textwidth]{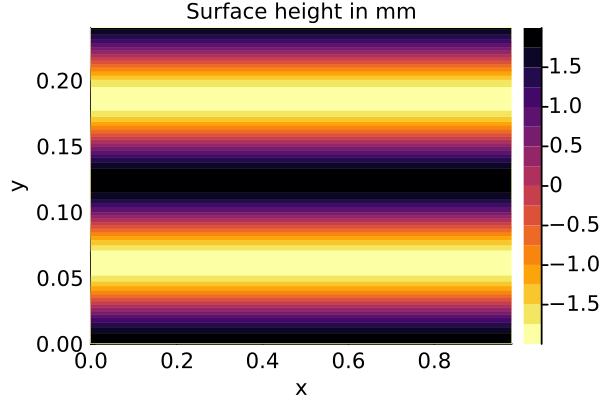}
        \caption{Surface height.}
        \label{fig_4_soil}
     \end{subfigure} \hspace{-3mm}
     \begin{subfigure}[b]{0.33\textwidth}
        \includegraphics[width=\textwidth]{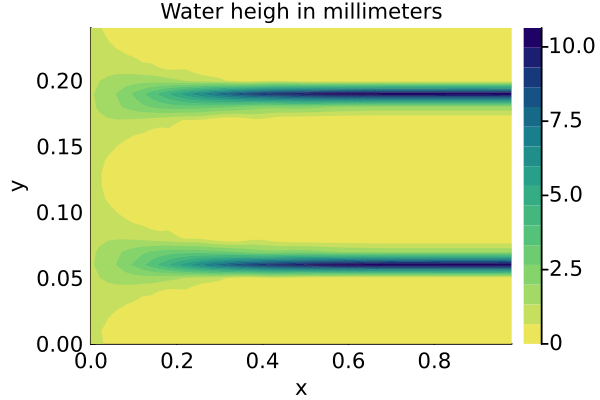}
        \caption{Water height.}
        \label{fig_4_water}
     \end{subfigure} \hspace{0mm}
     \begin{subfigure}[b]{0.33\textwidth}
        \includegraphics[width=\textwidth]{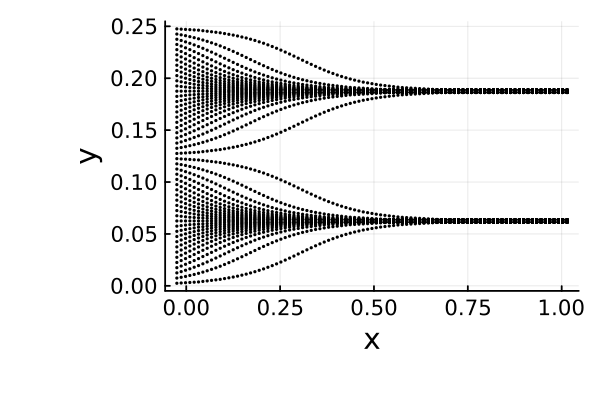}
        \caption{Particle positions.}
        \label{fig_4_part}
     \end{subfigure}
     \caption{Test case with dry areas: surface height~(a), water height~(b) and positions of the particles~(c). $z=z_3$, $r = 0$ and $h = 4 h_{char}$, $\varepsilon= \varepsilon_{char}$, $\Delta s= 2 \Delta s_{char}$.}
     \label{fig_4}
\end{figure}

\noindent The solution of system \eqref{eq_syst_linear_h} presents some dry areas, on the highest areas of the surface. Figure~\ref{fig_4_water} presents the water height computed from the particle scheme, and figure~\ref{fig_4_part} the position of particles on the domain. We observe that they regroup in the lowest parts of the bottom surface, the so-called thalweg in hydrology terms. Between these thalwegs there are no particles, which means that the water height vanishes.

\subsection{Application to a landscape evolution model}

In this section we present numerical simulations of a landscape evolution model studied in \cite{binard2024well}.
The model consists of three partial differential equations on the bottom topography $z$, the fluid height $h$ and the sediment concentration $c$. Focusing on topography erosion, the time derivatives of $h,c$ are negligible and
the system is given by: $\forall t \geq 0$, $\forall x \in \Omega \subset \RR^2$, 
\begin{subequations}
\begin{empheq}[left= \empheqlbrace]{align}
& \displaystyle
\dive(hv) = r(t,x), \vspace{2mm} \label{eq_syst_statio_h} \\
&\displaystyle
\dive (ch v) = \rho_s \, e \left( \frac{h(t,x)}{H} \right)^m \left( \frac{|v(t,x)|}{V} \right)^n - \rho_s \, s \, \frac{c(t,x)}{c_{sat}},\vspace{2mm} \label{eq_syst_statio_c} \\
&\displaystyle
\dt z = K \Delta z(t,x) - e \left( \frac{h(t,x)}{H} \right)^m \left( \frac{|v(t,x)|}{V} \right)^n + s \, \frac{c(t,x)}{c_{sat}}, \label{eq_syst_statio_z} \\
& \displaystyle
v(t,x) = V \, (\tan \theta,0) - V \, \nabla (h+z) (t,x).
\end{empheq}
\label{syst_statio_method_part}
\end{subequations}

The function $r$  represents the water source term, $\rho_s$ is the volumetric mass density of the sediments, $c_{sat}$ is the concentration of saturation for sediments, $e$ is the erosion rate, $s$ is the sedimentation rate, and $K$ is the constant of creep. The constants $H$ and $V$ are characteristic values for water height and water velocity. 

\medskip
Equation~\eqref{eq_syst_statio_z} is discretized in time by an explicit Euler scheme. Denoting by $dt$ the time step of this scheme, we define $t_i := i \, dt$. Then, due to the expression of water velocity, both equations~\eqref{eq_syst_statio_h} and~\eqref{eq_syst_statio_c} are fully 
non linear. Therefore, these equations cannot be solved directly by the previous particle scheme, at $t \geq 0$ fixed. However, at each time step $t_i$, the quantities $z_i := z(t_i,.)$ and $v_i := v(t_i,.)$ differ only by an order $O(dt)$ term from those at the previous time step $t_{i-1}$. Consequently, we can propose a consistent in time discretization of these equations by evaluating these quantities at the previous time step in the expressions of the water velocity and of the source terms. We consider the following time discretization of~\eqref{syst_statio_method_part}. It leads to a linear system for the unknowns $h_{i+1}$, $c_{i+1}$, $z_{i+1}$: $\forall i \in \NN$, $\forall x \in \Omega \subset \RR^2$,
\begin{subequations}
\begin{empheq}[left= \empheqlbrace]{align}
& \displaystyle
\dive(h_{i+1} v_{i})(x) = r_{i+1}(x), \vspace{2mm} \label{eq_syst_statio_hn} \\
&\displaystyle
\dive (c_{i+1} h_{i+1} v_{i}) (x) = \rho_s \, e \left( \frac{h_i(x)}{H} \right)^m \left( \frac{|v_i(x)|}{V} \right)^n - \rho_s \, s \, \frac{c_{i}(x)}{c_{sat}},\vspace{2mm} \label{eq_syst_statio_cn} \\
&\displaystyle
z_{i+1}(x) = z_i(x) + dt K \Delta z_i(x) - e \, dt \left( \frac{h_i(x)}{H} \right)^m \left( \frac{|v_i(x)|}{V} \right)^n + s \, dt \, \frac{c_i(x)}{c_{sat}}, \\ \label{eq_syst_statio_zn} 
& \displaystyle
v_i(x) = V \, (\tan \theta,0) - V \, \nabla (h_i+z_i) (x), \nonumber
\end{empheq}
\label{syst_statio_method_part_discr_time}
\end{subequations}
\noindent
\hspace{-1mm}where $r_i(x) \approx r(t_i,x)$, $h_i(x) \approx h(t_i,x)$, and $c_i(x)\approx c(t_i,x)$.
The first two equations of this system, equation.~\eqref{eq_syst_statio_hn} and equation~\eqref{eq_syst_statio_cn} can be solved by the Particle method. The scheme is initialized with a constant elevation, that is $h_0(x,y)+z_0(x,y) = h_0$. Therefore, the initial water velocity is constant: $v_0(x,y) = V (\tan \theta,0) $. The initial concentration $c$ is defined to be constant: $c_0(x,y) = c_0$.

\paragraph{Choice of parameters}

\begin{table}[ht]
\center 
\begin{tabular}{| l| l | l |}
     \hline
     Length of the domain & $L_x$ & $40$ cm\\
     Width of the domain & $L_y$ & $10$ cm\\
     Characteristic water speed & $ V$ & $1$ m/s\\	
     Exponent of friction over $h$ & $m$ & $1,6$\\
     Exponent of friction over $v$ & $n$ & $3,2$\\
    Density of the sediments & $\rho_s$ & $2,17.10^6$ g/m$^3$ \\
    Concentration at saturation & $c_{sat}$ & $3,17.10^5$ g/m$^3$ \\
    Erosion speed & $e$ & $0,5$ mm/hour\\
    Sedimentation speed & $s$ & $e/2000$ \\
    Angle of the plane & $\theta$ & $39^\circ$\\
    Water height at the inbound boundary & $h_0$ & $0.5$ mm\\
    Concentration of sediments at the inbound boundary & $c_0$ & $317$ g/m$^3$\\
     \hline  
\end{tabular}
\caption{Parameters of the model.}
\label{tab_param_landscape}
\end{table}

\begin{table}[ht]
\center 
\begin{tabular}{| l| l | l |}
    \hline
    Number of points of the grid in the $x$ direction & $N_x$ & $400$\\
    Number of points of the grid in the $y$ direction & $N_y$ & $100$\\
    Time step for the equation on $z$ & $dt$ & $4.5$ s\\
    Kernel diameter & $h$ & $L_y/4$\\
    Distance between particles at the boundary & $\varepsilon$ & $h/10$\\
    Discretisation of variable $s$ & $\Delta s$ & $0.005$ s\\
    \hline  
\end{tabular}
\caption{Parameters of numerical simulations}
\label{tab_param_num_landscape}
\end{table}

The choice of parameters is the same as in \cite{binard2024well}, where numerical simulations of system \eqref{syst_statio_method_part} were carried out with a finite-volume method.
It is given in Table~\ref{tab_param_landscape}.
The choice of numerical parameters is given in Table \ref{tab_param_num_landscape}. The computation of $\nabla z$ and $\Delta z$ is performed with a finite difference discretization on a grid of size $N_x \times N_y$.

\begin{figure}[ht]
     \centering
     \begin{subfigure}[b]{0.42\textwidth}
        \includegraphics[width=\textwidth]{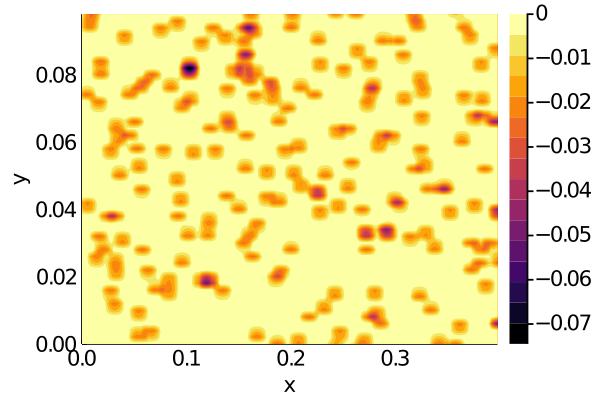}
        \caption{Initial surface $z$, in millimeters.}
        \label{fig_init_z}
    \end{subfigure} \hspace{1mm}
    \begin{subfigure}[b]{0.42\textwidth}
        \includegraphics[width=\textwidth]{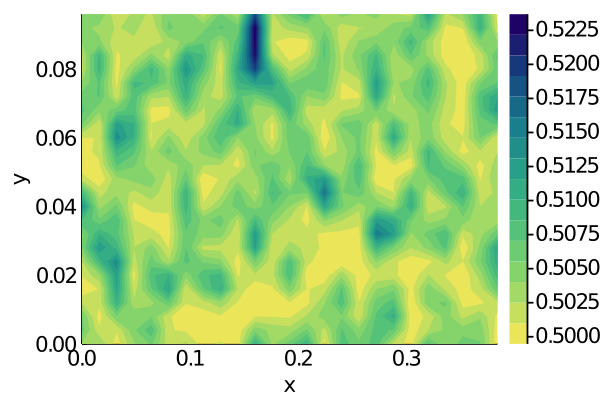}
        \caption{Initial water height, in millimeters.}
        \label{fig_init_water}
    \end{subfigure}
    \caption{Heat maps of the initial surface~(A) and initial water height~(B), in the domain $[0,L_x] \times [0,L_y]$.}
    \label{fig_sol_init_part}
\end{figure}

In the numerical simulations, the initial surface is a flat tilted plane with a small random perturbation, in the form of channels dug at random positions. This surface, seen from above is depicted in Figure~\ref{fig_sol_init_part}. The computation time is about three hours for $1000$ time steps, that is a final time of $1.25$ hours in the simulation.

\paragraph{Results of numerical simulations}

\begin{figure}[ht]
	\centering
     \begin{subfigure}[t]{0.42\textwidth}
          \centering
          \includegraphics[width=\textwidth]{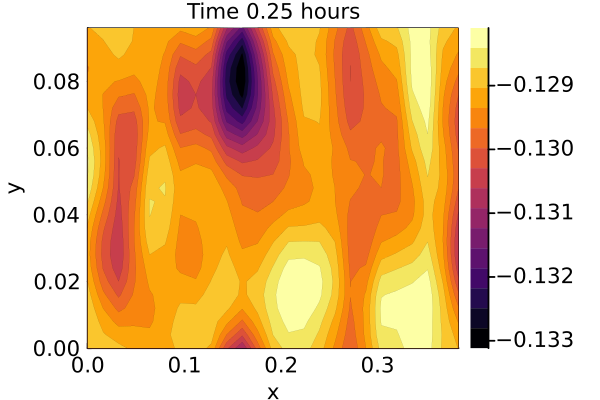}
          \caption{Surface height after $200$ time steps.}
          \label{fig_stable_soil_200}
     \end{subfigure} \hspace{0mm}
     \begin{subfigure}[t]{0.42\textwidth}
          \centering
          \includegraphics[width=\textwidth]{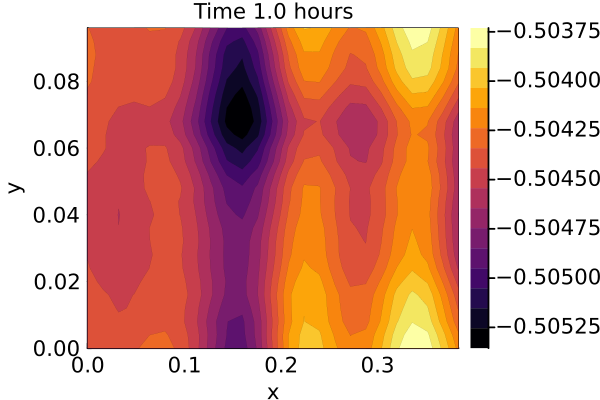}
          \caption{Surface height after $800$ time steps.}
          \label{fig_stable_soil_800}
    \end{subfigure}\\ \vspace{2mm}
	\begin{subfigure}[t]{0.42\textwidth}
        \includegraphics[width=\textwidth]{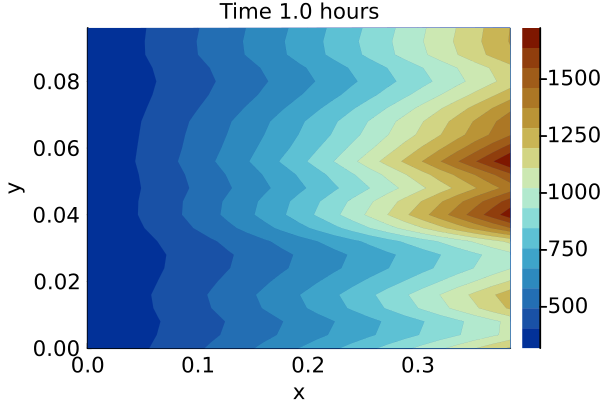}
        \caption{Sediment concentration  after $800$ time steps.}
        \label{fig_stable_c}
    \end{subfigure} \hspace{2mm}
    \begin{subfigure}[t]{0.42\textwidth}
        \includegraphics[width=\textwidth]{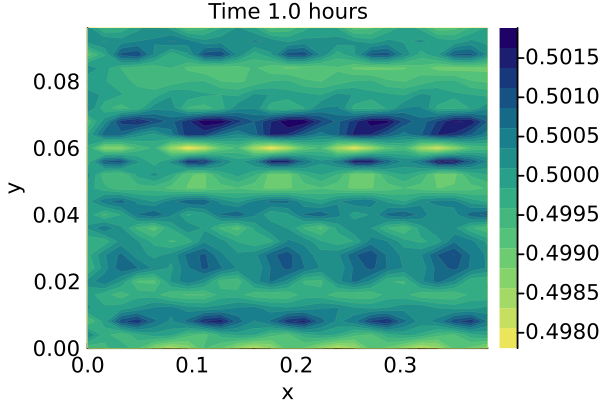}
        \caption{Water height after $800$ time steps.}
        \label{fig_stable_h}
    \end{subfigure}
    \caption{Heat maps of the surface height in millimeters $K=K_e$ after $200$ time steps~(a) and $800$ time steps~(b), sediment concentration in grams per cubic meters~(c) and water height in millimeters~(d) both after $800$ time steps, in the domain $[0,L_x] \times [0,L_y]$, for $K=K_e$. Scales are given by the color bar to the right of each picture.}
    \label{fig_stable}
\end{figure}

The constant of creep $K$ plays an important role in the behavior of the system. If $K$ is  large enough, the system is stable around the stationary solution of the flat plane, and when it is smaller the system is unstable, and channels are dug on the surface.
The stability analysis of the system is performed in \cite{binard2024well}. First, we show some numerical results with the creep parameter $K$:
$$
\displaystyle
K=K_e=\frac{5\times 10^{-4}}{3600}m^2s^{-1},
$$ 
which correspond to a stable case.
\medskip

Figures~\ref{fig_stable_soil_200} and~\ref{fig_stable_soil_800} show the surface height $z$, after $200$ and $800$ time steps. We observe that the initial perturbations have been flattened. Indeed, after $800$ time steps the amplitude of the variations of the surface height is $2 \%$ of the initial amplitude.
Then, Figure~\ref{fig_stable_c} shows the concentration of sediments in water, after $800$ time steps. The concentration at the left of the domain is the same order as $c_0 = 317$g/m$^3$, and is larger on the right due to erosion. Figure~\ref{fig_stable_h} shows the water height after $800$ time steps of the scheme. Its value is almost constant, around $0.5$mm. \newline

\begin{figure}[ht]
	\centering
     \begin{subfigure}[t]{0.45\textwidth}
          \centering
          \includegraphics[width=\textwidth]{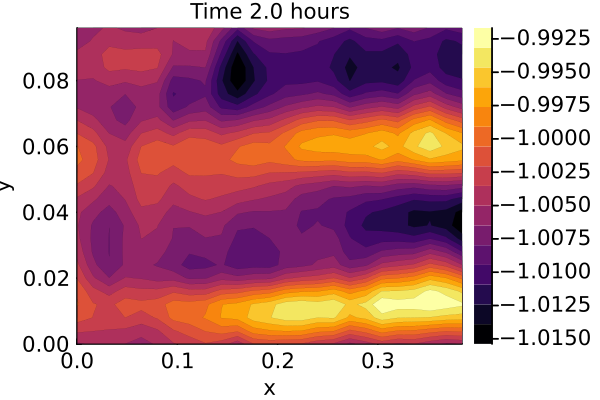}
          \caption{Surface height for $K = \frac{K_e}{20}$.}
          \label{fig_part_sol_unstable_20}
     \end{subfigure}\hspace{5mm}
     \begin{subfigure}[t]{0.45\textwidth}
          \centering
          \includegraphics[width=\textwidth]{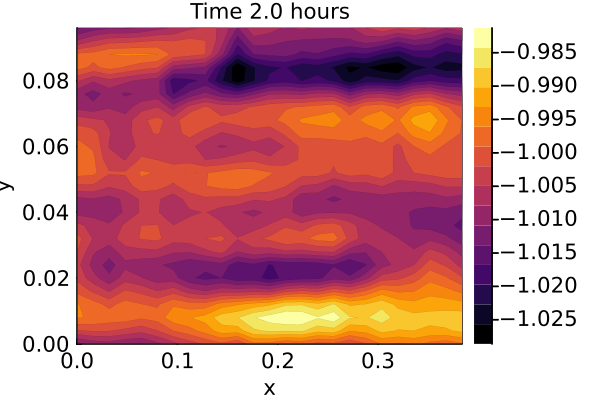}
          \caption{Surface height for $K = \frac{K_e}{50}$.}
          \label{fig_part_sol_unstable_50}
     \end{subfigure}
     \caption{Heat maps of the surface height in millimeters for $K=K_e/20$~(A) and $K=K_e/50$~(B) after $1600$ time steps, in the domain $[0,L_x] \times [0,L_y]$.}
     \label{fig_part_sol_unstable}
\end{figure}
     
\begin{figure}[ht]
	\centering
     \begin{subfigure}[t]{0.325\textwidth}
          \centering
          \includegraphics[width=\textwidth]{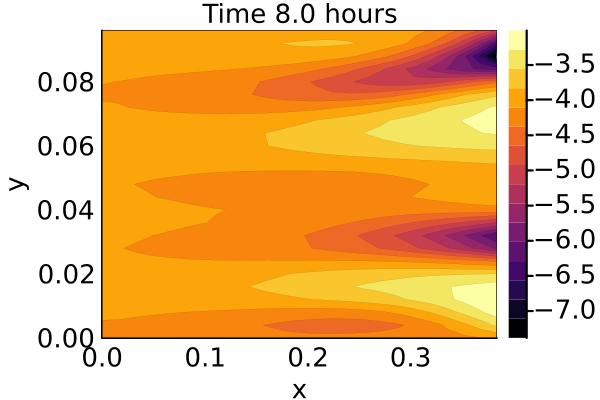}
          \caption{Surface height.}
          \label{fig_z_unstable_later}
     \end{subfigure} \hspace{-0mm}
     \begin{subfigure}[t]{0.325\textwidth}
          \centering
          \includegraphics[width=\textwidth]{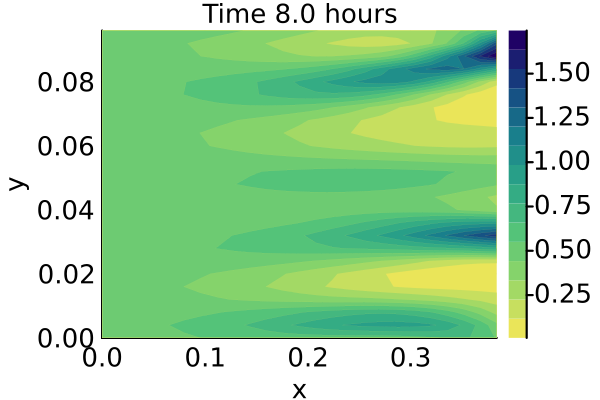}
          \caption{Water height.}
          \label{fig_h_unstable_later}
     \end{subfigure} \hspace{-1mm}
     \begin{subfigure}[t]{0.325\textwidth}
          \centering
          \includegraphics[width=\textwidth]{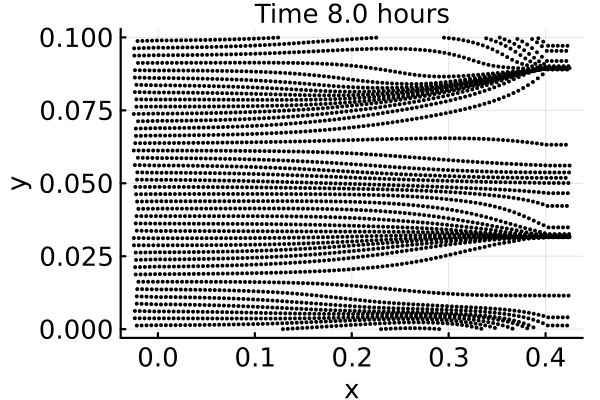}
          \caption{Particles positions.}
          \label{fig_part_unstable_later}
     \end{subfigure}
     \caption{Heat maps of the surface height~(A) water height~(B) in millimeters and particle positions~(C) for $K=K_e/50$ after $6400$ time steps, in the domain $[0,L_x] \times [0,L_y]$.}
     \label{fig_part_sol_unstable2}
\end{figure}
In Figure~\ref{fig_part_sol_unstable_20} and~\ref{fig_part_sol_unstable_50} we display the surface height, when $K=K_e/20$ and $K=K_e/50$. In these cases the system is unstable at some wavelengths, and we expect to observe a transverse instability. We can observe the formation of channels in the soil, in the flow direction. The amplitude of these channels is about $32 \%$ of the amplitude of the initial perturbations in the first case, and a bit more than $57 \%$ in the second case. 
Moreover, the depth of the channels is bigger when $K$ is smaller, in accordance with the stability results of \cite{binard2024well}.
After a sufficiently large number of time iterations with $K=K_e/50$, channels are deeper and water is concentrated inside these channels. Water height vanishes in highest areas, as in figure~\ref{fig_h_unstable_later}. We observe in figure~\ref{fig_part_unstable_later} that these areas do not contain particles. Particles gather inside the channels, so we could consider to merge some particles when they are too close, to improve the method. The surface height is represented in figure~\ref{fig_z_unstable_later}, and has variations of $4$ millimeters, which is $57$ times bigger than the amplitude of the initial random perturbations.\\

In \cite{binard2024well}, system \eqref{eq_landscape_model} was solved with a finite-volume method. This scheme can only be used with a positive water height. If the water height approaches zero at some points in the domain, some negative values of $h$ can appear in the computed solution. A finite-volume scheme that can handle dry areas could be designed, as in \cite{delestre2010simulation}, but require a more sophisticated method. Our Particle method is able to cope with the appearance of regions of vanishing water height (these areas simply do not contain any particles). Consequently, this Particle method can be used on more general landscapes than the finite-volume method of~\cite{binard2024well}, in particular, landscapes which may contain valleys and mountains.

\paragraph{Convergence tests}

\begin{figure}[ht]
	\centering
     \begin{subfigure}[t]{0.45\textwidth}
          \centering
          \includegraphics[width=\textwidth]{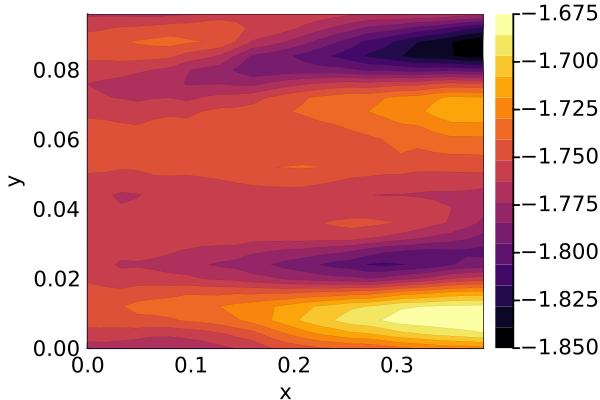}
          \caption{$dt=4.5$ s, $\varepsilon=h/10$ cm and $\Delta s = 0.005$ s, $2800$ time steps.}
          \label{fig_part_sol_unstable_compar1}
     \end{subfigure}\\
     \begin{subfigure}[t]{0.45\textwidth}
          \centering
          \includegraphics[width=\textwidth]{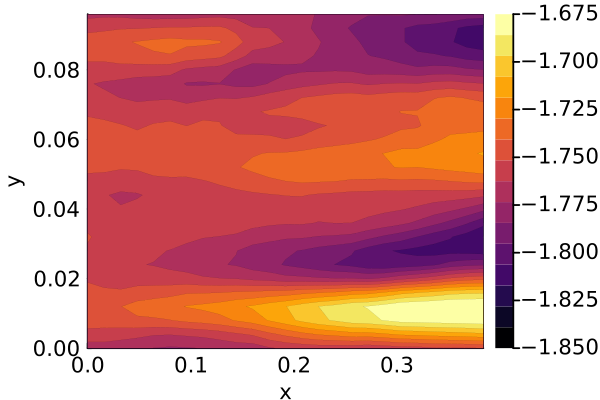}
          \caption{$dt=2.25$ s, $\varepsilon=h/10$ cm and $\Delta s = 0.005$ s, $5600$ time steps.}
          \label{fig_part_sol_unstable_compar_dts_petit}
     \end{subfigure}
     \begin{subfigure}[t]{0.45\textwidth}
          \centering
          \includegraphics[width=\textwidth]{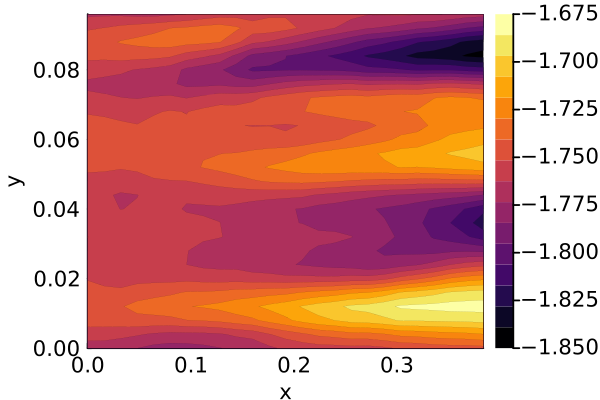}
          \caption{$dt=9$ s, $\varepsilon=h/10$ cm and $\Delta s = 0.005$ s, $1400$ time steps.}
          \label{fig_part_sol_unstable_compar_dts_grand}
     \end{subfigure}
     \caption{Heat maps of the surface height in millimeters for $K=K_e/50$ at time $3.5$ hours, in the domain $[0,L_x] \times [0,L_y]$. Parameters of table~\ref{tab_param_num_landscape}~(A), smaller time step~(B), bigger time step~(C).}
     \label{fig_part_sol_unstable_compar_1}
\end{figure}

\begin{figure}[hbtp]
	\centering
     \begin{subfigure}[t]{0.45\textwidth}
          \centering
          \includegraphics[width=\textwidth]{Images_part/z_part_K=e50_etape2800.png}
          \caption{$dt=4.5$ s, $\varepsilon=h/10$ cm and $\Delta s = 0.005$ s.}
          \label{fig_part_sol_unstable_compar12}
     \end{subfigure}\\[3mm]
     \begin{subfigure}[t]{0.45\textwidth}
          \centering
          \includegraphics[width=\textwidth]{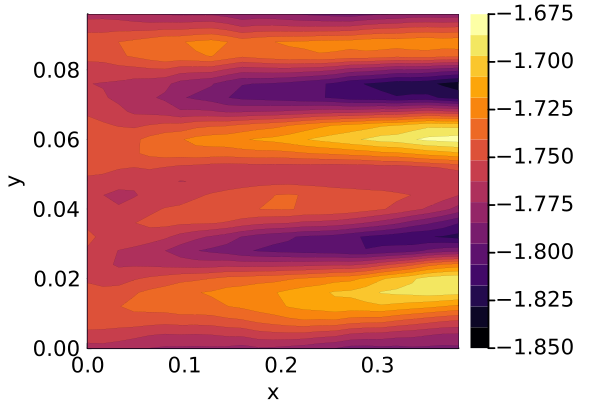}
          \caption{$dt=4.5$ s, $\varepsilon=h/5$ cm and $\Delta s = 0.005$ s.}
          \label{fig_part_sol_unstable_compar_eps_grand}
     \end{subfigure}
     \hspace{2mm} 
     \begin{subfigure}[t]{0.45\textwidth}
          \centering
          \includegraphics[width=\textwidth]{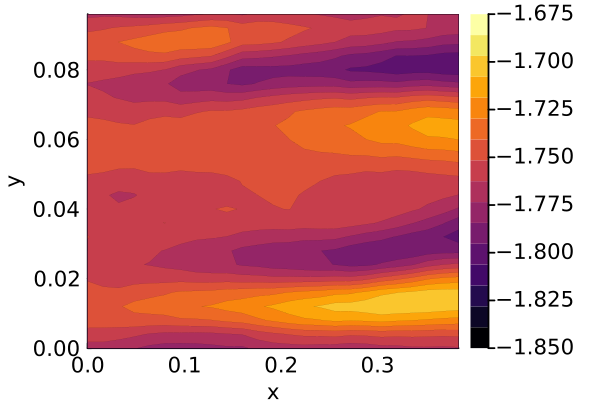}
          \caption{$dt=4.5$ s, $\varepsilon=h/20$ cm and $\Delta s = 0.005$ s.}
          \label{fig_part_sol_unstable_compar_eps_petit}
     \end{subfigure}\\[3mm]
     \begin{subfigure}[t]{0.45\textwidth}
          \centering
          \includegraphics[width=\textwidth]{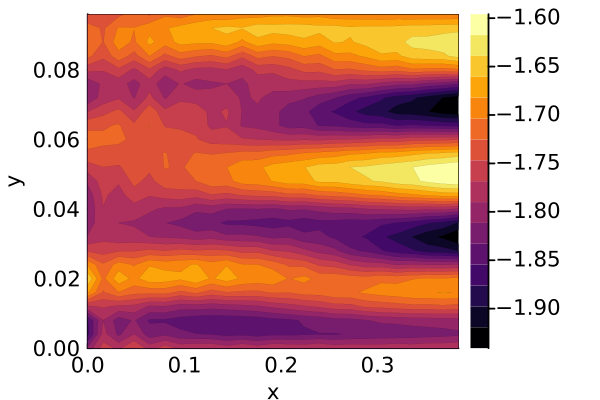}
          \caption{$dt=4.5$ s, $\varepsilon=h/10$ cm and $\Delta s = 0.01$ s. Different color bar scale.}
          \label{fig_part_sol_unstable_compar_ds_grand}
     \end{subfigure}
     \hspace{2mm} 
     \begin{subfigure}[t]{0.45\textwidth}
          \centering
          \includegraphics[width=\textwidth]{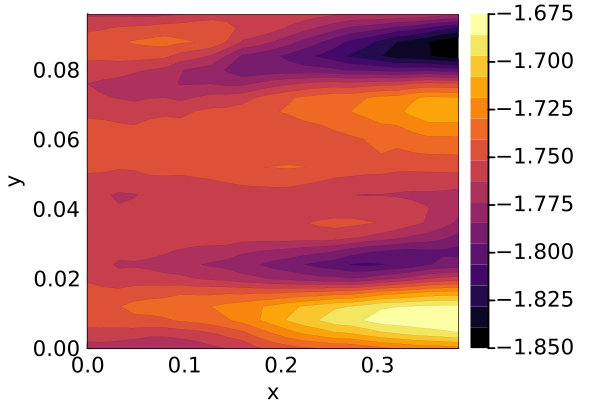}
          \caption{$dt=4.5$ s, $\varepsilon=h/10$ cm and $\Delta s = 0.0025$ s.}
          \label{fig_part_sol_unstable_compar_ds_petit}
     \end{subfigure}
     \caption{Heat maps of the surface height in millimeters for $K=K_e/50$ at time $3.5$ hours ($2800$ time steps), in the domain $[0,L_x] \times [0,L_y]$. Scales are the same for all pictures except (D) which is larger. Parameters of table~\ref{tab_param_num_landscape}~(A), fewer/more particles in the transverse direction~(B)/(C), fewer/more particles in the longitudinal direction~(D)/(E).}
     \label{fig_part_sol_unstable_compar_2}
\end{figure}

Figures~\ref{fig_part_sol_unstable_compar_1} and~\ref{fig_part_sol_unstable_compar_2} show results of the simulations with $K = K_e/50$ at time $3.5$ hours, for different values of the numerical parameters. In Figure~\ref{fig_part_sol_unstable_compar_1}, the time step of the landscape evolution model vary. We observe similar behavior between the initial figure~\ref{fig_part_sol_unstable_compar1}, when the time step is reduced~\ref{fig_part_sol_unstable_compar_dts_petit} and when it is increased~\ref{fig_part_sol_unstable_compar_dts_grand}.

Then, when the number of particles on the domain is reduced in the transverse or the longitudinal direction ($\varepsilon$ or $\Delta_s$ bigger) as shown in Figures~\ref{fig_part_sol_unstable_compar_eps_grand} and\ref{fig_part_sol_unstable_compar_ds_grand} the channels are shallower. On the other hand, if the number of particles on the domain is increased ($\varepsilon$ or $\Delta s$ smaller) in Figures~\ref{fig_part_sol_unstable_compar_eps_petit} and \ref{fig_part_sol_unstable_compar_ds_petit}, the solution is quite similar to the reference solution \ref{fig_part_sol_unstable_compar12}.

\section{Conclusion and perspectives}

Motivated by applications in the simulation of landscape evolution models, we have proposed a new method to solve scalar equations in a divergence form that preserves the positivity of the solutions. This is a Particle method and, to our knowledge, this is the first time it is applied to solve stationary problems. We have proven the convergence of the scheme under some regularity assumptions on the data (which in turn guarantee the smoothness of the solution), under some (classical) restrictions on the numerical parameter and under the assumption that the characteristic curves fulfill the whole computational domain. This later assumption is restrictive, but we observe that the numerical solution is well behaved even in the presence of dry areas. When the source term vanishes in these areas, one can extend the solution by assuming it vanishes also. We have performed several convergence tests which confirm our convergence result. We have also applied our numerical scheme in the context of landscape evolution models when fluid height vanishes and standard finite-volume methods may be numerically unstable.


Regarding applications to the simulation of landscape evolution model, some further developments are needed to make our scheme more efficient or applicable to more general landscapes. First, this scheme is very costly compared to a finite-volume scheme. For the simulation of the landscape evolution model, the computation time of one time step is $200$ times shorter with the finite-volume scheme than with the particle scheme, to obtain a similar precision on the solutions. This is due to the non-linear terms depending on $\nabla h$ which have to be computed by a sum over the neighboring particles. However, note that the computation of the Particle method can be parallelized, as the computations for each trajectory can be run on independent processors. 
From a numerical point of view, we can also improve the accuracy of the scheme in the non-linear case: here, we proposed a linearized version of the problem at each time step which may generate a loss of accuracy. One may consider also solving directly the non-linear problem at each time step by a Newton method or a fixed point procedure.
Note also that particles computed as solutions to the linear system tend to gather in the lowest parts of the surface. If they are too close, then the gradient of $h$ can become singular, which leads to considerable errors in the computation of the solution. Consequently, the time steps have to be taken small enough to obtain good convergence. One could also consider merging particles that are too close to prevent the formation of singularities.


\appendix

\section{Proof of proposition~\ref{prop_ex_flow}} \label{appendix_proof_prop_assump}

\begin{proof}
Items $(i)$ and $(ii)$ of the theorem are proved by a direct application of the Cauchy-Lipschitz theorem. Then we prove Item $(iii)$. First, as ${\bf a}$ is $P$-Lipschitz continuous, by \eqref{eq_a_incoming} and the fact that $\Phi(0,\xi) = (0,\xi)$, we have for $t\geq 0$:
\begin{align*}
\displaystyle
|\Phi(t,\xi) - \Phi(0,\xi)| 
&\leq \int_0^t \left|{\bf a}(\Phi(s,\xi)) \right| ds 
\leq t \beta + \int_0^t |{\bf a}(\Phi(s,\xi))- {\bf a}(\Phi(0,\xi))| ds\\
& \leq t \beta + P \int_0^t |\Phi(s,\xi) - \Phi(0,\xi) | ds.
\end{align*}
A similar inequality can be shown for $t\leq 0$. Then, by Gronwall's lemma, denoting by $\Phi_1(t,\xi)$ the first component of $\Phi(t,\xi)$, and using that $\Phi_1(0,\xi)=0$ we get:
$$|\Phi_1(t,\xi)| \leq \frac{\beta}{P} \left( e^{P|t|}-1 \right).$$
Thus for all $|t| \leq \ell := \frac{1}{P}\ln(1+\frac{L P}{\beta})$, we have $|\Phi_1(t,\xi)| \leq L$. This implies that $\Phi_1(t,\xi) \geq -L$ for $t \in (-\ell, 0]$. Then for $t >0$ we show that $\Phi_1(t,\xi)>0$. Indeed, because of~\eqref{eq_a_incoming} we have $\Phi_1(t,\xi)>0$ for $t \in (0,\varepsilon_0]$ with $\varepsilon_0>0$ small. By way of contradiction, suppose that $t_1>0$ is the smallest time such that $\Phi_1(t_1,\xi)=0$. So we have $\Phi_1(t,\xi) >0$ for $t \in [t_1-\varepsilon_1,t_1)$ with $\varepsilon_1>0$ small enough. This contradicts the fact that 
$$
\displaystyle
\dt \Phi_1(t_1,\xi) = {\bf a}_1(\Phi(t_1,\xi)) = {\bf a}_1(0, \Phi_2(t_1,\xi), \dots, \Phi_d(t_1,\xi)) > 0,
$$
by~\eqref{eq_a_incoming}. Hence $\Phi_1(t,\xi) > -L$, $\forall (t,\xi) \in (-\ell,+\infty) \times \RR^{d-1}$. This shows the item $(iii)$ (which is also Assumption~\ref{assump_flow}~(i)).\\

Finally we prove Item $(iv)$. We first show that with the sole property~\eqref{eq_a_incoming}, there exists $\delta > 0$ such that $(-\delta,\delta) \times \RR^{d-1} \subset \Phi(\Omega_\ell)$. Indeed, since ${\bf a }$ is $P$-Lipschitz continuous and with~\eqref{eq_a_incoming}, setting $\delta = \frac{\alpha}{2}\min \left( \frac{1}{P}, \ell \right)$ we have:
\begin{align*}
 {\bf a }_1 (x_1,\eta) \geq \alpha - P |x_1| > \frac{\alpha}{2}, \quad \forall (x_1,\eta) \in (-\delta,\delta) \times \RR^{d-1}.
\end{align*}
Now, for such $(x_1,\eta)$, define $\chi(t)$ as the solution of
\begin{align}
\chi'(t) = {\bf a} (\chi(t)), \quad \chi(0) = (x_1,\eta).
\label{eq_lemma_caract_eta}
\end{align}
Without loss of generality, suppose $x_1 \in (-\delta,0)$. We have 
\begin{align}
    \chi_1(t_1) - \chi_1(t_0) = \int_{t_0}^{t_1} {\bf a}_1 (\chi(s)) ds
    \geq \frac{\alpha}{2}(t_1-t_0) >0,
    \label{eq_lemma_caract_int}
\end{align}
for all $t_1 > t_0 > 0$ provided that $\chi_1(s) \in (-\delta,\delta)$ , $\forall s \in [0,t_1]$. Therefore $\chi_1$ is strictly increasing on some interval $(0,T_1]$, with $T_1>0$. Since $0 \in (-\delta, \delta)$, there exists a time $t_2>0$ such that $0<t_2<T_1$ and $\chi_1(t_2)=0$. Furthermore, using \eqref{eq_lemma_caract_int} with $(t_0,t_1) = (0,t_2)$, $t_0$ satisfies 
$$\delta > -x_1 = \int_0^{t_2} {\bf a}_1 (\chi(s)) ds > \frac{\alpha t_2}{2},$$
hence $t_2 < 2\delta / \alpha \leq \ell$. We denote $\chi(t_2) = (0,\xi)$. Then, by the uniqueness in the Cauchy-Lipschitz theorem we have $\Phi(t-t_2,\xi) = \chi(t)$ for all $t \in (0, t_2)$. Thus for $t=0$, one has:
$\Phi(-t_2,\xi) =\chi(0)=(x_1, \eta)$. This shows that $(-\delta,\delta) \times \RR^{d-1} \subset \Phi(\Omega_\ell)$.\\

Now we assume that ${\bf a}_1 (x_1,\xi) \geq \alpha >0$ for all $(x_1,\xi) \in U$ and consider $(x_1,\eta) \in [\delta,+\infty) \times \RR^{d-1}$. Let us consider the solution $\chi(t)$ of \eqref{eq_lemma_caract_eta}. 
For $t>0$, we have:
\begin{align*}
    |\chi(-t) - \chi(0)| 
    = |\int_{0}^{t} \left( {\bf a} (x_1,\eta) + {\bf a} (\chi(-s)) - {\bf a} (\chi(0)) \right) ds|
    \leq | {\bf a} (x_1,\eta)|t + P \int_0^t |\chi(-s) - \chi(0) | ds.
\end{align*}
Hence, by Gronwall's lemma, 
$$|\chi(-t)-\chi(0)| \leq \frac{|{\bf a} (x_1,\eta)|}{P} (e^{Pt}-1).$$
This shows that $\chi(-t)$ is defined for all $t>0$ as long as $\chi(-t) \geq -L$. Thus, there exists $t_0 >0$ such that $\chi_1(-t_0)=0$. Setting $\chi(-t_0) = (0,\xi)$, we have $\Phi(t_0,\xi) = (x_1,\eta)$ which finishes the proof of Assumption~\eqref{assump_flow}~(ii).
\end{proof}

\bibliographystyle{plain}
\bibliography{biblio_particle.bib}

@article{ben2000,
  title={Convergence of {SPH} method for scalar nonlinear conservation laws},
  author={Ben Moussa, B and Vila, JP},
  journal={SIAM Journal on Numerical Analysis},
  volume={37},
  number={3},
  pages={863--887},
  year={2000},
  publisher={SIAM}
}

@article{binard2024well,
  title={Well-posedness and stability analysis of a landscape evolution model},
  author={Binard, Julie and Degond, Pierre and Noble, Pascal},
  journal={Journal of Nonlinear Science},
  volume={34},
  number={1},
  pages={20},
  year={2024},
  publisher={Springer}
}

@article{chen_landscape_2014,
  title={Landscape evolution models: A review of their fundamental equations},
  author={Chen, Alex and Darbon, J{\'e}r{\^o}me and Morel, Jean-Michel},
  journal={Geomorphology},
  volume={219},
  pages={68--86},
  year={2014},
  publisher={Elsevier}
}

@article{chen_equations_2014,
  title={On the equations of landscape formation},
  author={Chen, Alex and Darbon, J{\'e}r{\^o}me and Buttazzo, Giuseppe and Santambrogio, Filippo and Morel, Jean-Michel},
  journal={Interfaces and Free Boundaries},
  volume={16},
  number={1},
  pages={105--136},
  year={2014}
}

@article{chorin1973discretization,
  title={Discretization of a vortex sheet, with an example of roll-up},
  author={Chorin, Alexandre Joel and Bernard, Peter S},
  journal={Journal of Computational Physics},
  volume={13},
  number={3},
  pages={423--429},
  year={1973},
  publisher={Elsevier}
}

@article{davis_convex_1892,
  title={The convex profile of bad-land divides},
  author={Davis, WM},
  journal={Science},
  number={508},
  pages={245--245},
  year={1892},
  publisher={American Association for the Advancement of Science}
}

@phdthesis{delestre2010simulation,
  title={Simulation du ruissellement d'eau de pluie sur des surfaces agricoles},
  author={Delestre, Olivier},
  year={2010},
  school={Universit{\'e} d'Orl{\'e}ans}
}

@book{dieudonne1969elements,
  title={El{\'e}ments d'analyse: tome I},
  author={Dieudonn{\'e}, Jean},
  year={1969},
  publisher={Gauthier-Villars}
}

@article{dominguez2011neighbour,
  title={Neighbour lists in smoothed particle hydrodynamics},
  author={Dom{\'i}nguez, JM and Crespo, AJC and G{\'o}mez-Gesteira, M and Marongiu, JC2869628},
  journal={International Journal for Numerical Methods in Fluids},
  volume={67},
  number={12},
  pages={2026--2042},
  year={2011},
  publisher={Wiley Online Library}
}

@article{evans1957particle,
  title={The particle-in-cell method for hydrodynamic calculations},
  author={Evans, Martha W and Harlow, Francis Harvey and Bromberg, Eleazer},
  year={1957},
  publisher={Los Alamos Scientific Laboratory Los Alamos}
}

@article{gilbert_convexity_1909,
  title={The convexity of hilltops},
  author={Gilbert, Grove Karl},
  journal={The Journal of Geology},
  volume={17},
  number={4},
  pages={344--350},
  year={1909},
  publisher={University of Chicago Press}
}

@article{gingold1977smoothed,
  title={Smoothed particle hydrodynamics: theory and application to non-spherical stars},
  author={Gingold, Robert A and Monaghan, Joseph J},
  journal={Monthly notices of the royal astronomical society},
  volume={181},
  number={3},
  pages={375--389},
  year={1977},
  publisher={Oxford University Press Oxford, UK}
}

@article{guerin_streamwise_2020,
  title={Streamwise dissolution patterns created by a flowing water film},
  author={Gu{\'e}rin, Adrien and Derr, Julien and Du Pont, Sylvain Courrech and Berhanu, Michael},
  journal={Physical Review Letters},
  volume={125},
  number={19},
  pages={194502},
  year={2020},
  publisher={APS}
}

@article{howard_channel_1983,
  title={Channel changes in badlands},
  author={Howard, Alan D and Kerby, Gordon},
  journal={Geological Society of America Bulletin},
  volume={94},
  number={6},
  pages={739--752},
  year={1983},
  publisher={Geological Society of America}
}

@article{leonard1985computing,
  title={Computing three-dimensional incompressible flows with vortex elements},
  author={Leonard, Anthony},
  journal={Annual Review of Fluid Mechanics},
  volume={17},
  number={1},
  pages={523--559},
  year={1985},
  publisher={Annual Reviews 4139 El Camino Way, PO Box 10139, Palo Alto, CA 94303-0139, USA}
}

@article{lucy1977numerical,
  title={A numerical approach to the testing of the fission hypothesis},
  author={Lucy, Leon B},
  journal={Astronomical Journal, vol. 82, Dec. 1977, p. 1013-1024.},
  volume={82},
  pages={1013--1024},
  year={1977}
}

@article{mas1987,
  title={A particle method for first-order symmetric systems},
  author={Mas-Gallic, S and Raviart, PA},
  journal={Numerische Mathematik},
  volume={51},
  number={3},
  pages={323--352},
  year={1987},
  publisher={Springer}
}

@article{monaghan1992smoothed,
  title={Smoothed particle hydrodynamics},
  author={Monaghan, Joe J},
  journal={Annual review of astronomy and astrophysics},
  volume={30},
  number={1},
  pages={543--574},
  year={1992},
  publisher={Annual Reviews 4139 El Camino Way, PO Box 10139, Palo Alto, CA 94303-0139, USA}
}

@incollection{raviart2006analysis,
  title={An analysis of particle methods},
  author={Raviart, Pierre-Arnaud},
  booktitle={Numerical Methods in Fluid Dynamics: Lectures given at the 3rd 1983 Session of the Centro Internationale Matematico Estivo (CIME) held at Como, Italy, July 7--15, 1983},
  pages={243--324},
  year={2006},
  publisher={Springer}
}

@article{rosenhead1931formation,
  title={The formation of vortices from a surface of discontinuity},
  author={Rosenhead, Louis},
  journal={Proceedings of the Royal Society of London. Series A, Containing Papers of a Mathematical and Physical Character},
  volume={134},
  number={823},
  pages={170--192},
  year={1931},
  publisher={The Royal Society London}
}

@article{evers2018continuum,
  title={From continuum mechanics to SPH particle systems and back: Systematic derivation and convergence},
  author={Evers, Joep HM and Zisis, Iason A and van der Linden, Bas J and Duong, Manh Hong},
  journal={ZAMM-Journal of Applied Mathematics and Mechanics/Zeitschrift f{\"u}r Angewandte Mathematik und Mechanik},
  volume={98},
  number={1},
  pages={106--133},
  year={2018},
  publisher={Wiley Online Library}
}

@article{lind2020review,
  title={Review of smoothed particle hydrodynamics: towards converged Lagrangian flow modelling},
  author={Lind, Steven J and Rogers, Benedict D and Stansby, Peter K},
  journal={Proceedings of the royal society A},
  volume={476},
  number={2241},
  pages={20190801},
  year={2020},
  publisher={The Royal Society Publishing}
}

\end{document}